\documentclass[a4paper, 12pt]{article}

\usepackage[sort&compress]{natbib}
\bibpunct{(}{)}{;}{a}{}{,} 

\usepackage{amsthm, amsmath, amssymb, mathrsfs, multirow, url, subfigure}
\usepackage{graphicx} 
\usepackage{ifthen} 
\usepackage{amsfonts}
\usepackage[usenames]{color}
\usepackage{fullpage}
\usepackage[normalem]{ulem}
\usepackage{tcolorbox}
\usepackage{accents}
\usepackage{soul}
\usepackage[ruled,vlined]{algorithm2e}
\usepackage{soul}

\theoremstyle{plain} 
\newtheorem{thm}{Theorem}
\newtheorem{cor}{Corollary}
\newtheorem{prop}{Proposition}

\newtheorem*{lem0}{Lemma}

\theoremstyle{definition}
\newtheorem{defn}{Definition}

\theoremstyle{remark} 
\newtheorem{remark}{Remark}
\newtheorem*{astep}{A--step}
\newtheorem*{pstep}{P--step}
\newtheorem*{cstep}{C--step}

\newcommand{\prob}{\mathsf{P}} 
\newcommand{\E}{\mathsf{E}}
\newcommand{\R}{\mathsf{R}}
\newcommand{\Q}{\mathsf{Q}}

\newcommand{\lPi}{\underline{\Pi}}
\newcommand{\uPi}{\overline{\Pi}}

\newcommand{\lgamma}{\underline{\gamma}}
\newcommand{\ugamma}{\overline{\gamma}}

\newcommand{\unif}{{\sf Unif}}

\newcommand{\A}{\mathcal{A}}

\newcommand{\I}{\mathscr{I}}
\newcommand{\RR}{\mathbb{R}}
 
\newcommand{\U}{\mathcal{U}}

\newcommand{\XX}{\mathbb{X}}
\newcommand{\YY}{\mathbb{Y}}
\newcommand{\UU}{\mathbb{U}}

\newcommand{\ZZ}{\mathbb{Z}}

\renewcommand{\S}{\mathcal{S}}

\newcommand{\plint}{C}

\newcommand{\model}{\mathscr{P}}

\title{Valid inferential models for prediction in supervised learning problems\footnote{This is an extended version of the 2021 International Symposium on Imprecise Probability Theory and Applications (ISIPTA) proceedings paper, \citet{CellaMartinISIPTA21}.}}

\author{Leonardo Cella\footnote{Department of Statistics, North Carolina State University; {\tt lolivei@ncsu.edu}, {\tt rgmarti3@ncsu.edu}} \quad and \quad Ryan Martin$^\dagger$}
\date{\today}

\begin{document}

\maketitle

\begin{abstract}
Prediction, where observed data is used to quantify uncertainty about a future observation, is a fundamental problem in statistics. Prediction sets with coverage probability guarantees are a common solution, but these do not provide probabilistic uncertainty quantification in the sense of assigning beliefs to relevant assertions about the future observable.  Alternatively, we recommend the use of a {\em probabilistic predictor}, a data-dependent (imprecise) probability distribution for the to-be-predicted observation given the observed data. It is essential that the probabilistic predictor be reliable or valid, and here we offer a notion of validity and explore its behavioral and statistical implications.  In particular, we show that valid probabilistic predictors must be imprecise, that they avoid sure loss, and that they lead to prediction procedures with desirable frequentist error rate control properties.  We  provide a general construction of a provably valid probabilistic predictor, which has close connections to the powerful conformal prediction machinery, and we illustrate this construction in regression and classification applications.  



\smallskip 

\end{abstract}

\section{Introduction}
\label{S:intro}


Data-driven prediction of future observations is a fundamental problem.  Here our focus is on applications where the data $Z=(X,Y)$ consists of explanatory variables $X \in \XX \subseteq \RR^d$, for some $d \geq 1$, and a response variable $Y \in \YY$. That is, we observe a collection $Z^n = \{Z_i=(X_i,Y_i): i=1,\ldots,n\}$ of $n$ pairs from an exchangeable process.  The two most common examples are {\em regression} and {\em classification}, where $\YY$ is an open subset and finite subset of $\RR$, respectively.  We consider both cases in what follows. The prediction problem corresponds to a case where we are given a value $x_{n+1}$ of the next explanatory variable $X_{n+1}$, and the goal is to predict the corresponding future response $Y_{n+1} \in \YY$. 

By ``prediction'' we mean quantifying uncertainty about $Y_{n+1}$ in a data-dependent way, i.e., depending on the observed data $Z^n$ and the given value $x_{n+1}$ of $X_{n+1}$.  One perspective on prediction uncertainty quantification is the construction of a suitable family of prediction sets representing collections of sufficiently plausible values for $Y_{n+1}$; see, e.g., \citet{Vovk:2005}, \citet{CAMPI2009382}, \citet{kuleshov2018accurate}, and Equation \eqref{eq:pred.set} below.  While prediction sets are practically useful, there are prediction-related tasks that they cannot perform, in particular, it cannot assign degrees of belief (or betting odds, etc.)~to all relevant assertions or hypotheses ``$Y_{n+1} \in A$,'' for $A \subseteq \YY$.  An alternative approach is to develop what we refer to here as a {\em probabilistic predictor}, i.e., a probability-like structure (precise or imprecise probability) defined on $\YY$, depending on $Z^n$ and $x_{n+1}$, designed to quantify uncertainty about $Y_{n+1}$ by directly assigning degrees of belief to relevant assertions.  The most common approach to probabilistic prediction is Bayesian, where a prior distribution for the model is specified and uncertainty is quantified by the posterior predictive distribution of $Y_{n+1}$, given $Z^n$ and $X_{n+1}=x_{n+1}$.  Other non-Bayesian approaches leading to predictive distributions include \citet{LawlessandFredette2005}, \citet{CoolenBayes},  \citet{wang.hannig.iyer.2012}, and  \citet{Vovk2018NonparametricPD}. 

Before moving forward, it is important to distinguish between uncertainty quantification with prediction sets and with probabilistic predictors.  One does not need a full (precise or imprecise) probability distribution to construct prediction sets and, moreover, sets derived from a probabilistic predictor are not guaranteed to satisfy the frequentist coverage probability property that warrants calling them genuine ``prediction sets.''  Therefore, the motivation for going through the trouble of constructing probabilistic predictor, Bayesian or otherwise, must be that there are important prediction-related tasks that prediction sets cannot satisfactorily handle.  In other words, the belief assignments provided by a (precise or imprecise) probability must be a high priority.  Strangely, however, the reliability of probabilistic predictors is only ever assessed in terms of (asymptotic) coverage probability properties of their corresponding prediction sets.  Our unique perspective is that, since belief assignments are a priority, there ought to be a way to directly assess the reliability of a probabilistic predictor's belief assignments.  


For prediction problems where only the (response) variables $Y_1,\ldots,Y_n$ are observed, \citet{CellaMartinConformal} introduced a notion of validity for probabilistic predictors.  Roughly, their validity condition requires that the subsets $A \subseteq \YY$ to which the probabilistic predictor tends to assign large numerical degrees belief are the same as those that tend to contain $Y_{n+1}$.  The point being that such a constraint ensures that the belief assignments made by the probabilistic predictor are not systematically misleading.  Here we extend their notion of validity to the case where explanatory variables are present, and the precise definitions are given below in Definitions~\ref{def:valid}--\ref{def:uvalid}. 
It turns out these notions of validity have some important consequences, imposing certain constraints on the mathematical structure of the probabilistic predictor. Indeed, we argue in Section~\ref{S:validity} (see, also, Corollary~\ref{prop:fct} in Section~\ref{S:implications}) that validity 
can only be achieved by probabilistic predictors that take the form of an imprecise probability distribution. Section~\ref{s:ilust} provides a preview of the formal definition of validity and offers empirical support for the claim that precise probabilistic predictors cannot be valid. 

After formally introducing these notions of validity in Section~\ref{S:validity}, we explore their behavioral and statistical consequences.  First, we show that even the weaker validity property in Definition~\ref{def:valid} implies that the probabilistic predictor avoids (a property stronger than) the familiar sure loss property in the imprecise probability literature, hence is not internally irrational from a behavioral point of view.  We go on to show that prediction-related procedures, e.g., tests and prediction regions, derived from (uniformly) valid probabilistic predictors control frequentist error probability.  The take-away message is that a (uniformly) valid probabilistic predictor provides the ``best of both worlds''---it simultaneously achieves desirable behavioral and statistical properties. 


Given the desirable properties of a valid probabilistic predictor, the natural question is {\em how to construct one?} The probabilistic predictor we construct here is largely based on the general theory of valid {\em inferential models} (IMs) as described in \citet{mainMartin, martinbook}.  Martin and Liu's construction usually assumes a parametric model but, here, we aim to avoid such strong assumptions. For this, we use a particular extension of the so-called {\em generalized IM} approach developed in \citet{martin2015, MARTIN2018105}.  The basic idea is that a link/association between observable data, quantities of interest, and an unobservable auxiliary variable with known distribution can be made without fully specifying the data-generating process.  In Section~\ref{S:IM}, we develop a valid IM construction that assumes only exchangeability of the observed data process, no parametric model assumptions required.  There, in Theorem~\ref{thm:valid}, we establish that this general IM-based probabilistic predictor construction achieves the (uniform) validity property.  The specifics of this construction are presented in Section~\ref{S:Regression}, in the context of regression. Section~\ref{S:Classification} considers the classification problem, and we show that the discreteness of $Y$ in classification problems may cause the IM random set output, from which the probabilistic predictor is derived, to be empty with positive probability. Two possible adjustments are provided, with the one based on suitably ``stretching'' the random set being most efficient.

An important observation is that parallels can be drawn between our proposed IM construction and the conformal prediction approach put forward in \citet{Vovk:2005} and elsewhere.  This is interesting for at least two reasons.  
\begin{itemize}
\item It demonstrates that one does not necessarily need ``new methods'' to construct probabilistic predictors to achieve the desired (uniform) validity property, just an appropriate re-interpretation of the output returned by certain existing methods.  In particular, our proposed IM construction returns a possibility measure whose contour function is the transducer derived from an appropriate conformal prediction algorithm.  Consequently, all we need is the corresponding conformal prediction algorithm to achieve our goals.  
\vspace{-2mm}
\item However, there would be a variety of different ways the conformal prediction algorithm could be re-interpreted as a probabilistic predictor, e.g., as a precise probability distribution or one of several different imprecise probability distributions.  Our developments here reveal that the appropriate re-interpretation, the one that leads to (uniform) validity, is by treating the conformal transducer as the contour function that defines a possibility measure. 
\end{itemize}
These points, along with some other concluding remarks, are given in Section~\ref{S:discuss}.




\section{Prediction validity: a preview}
\label{s:ilust}

To help clarify the difference between the traditional notions of uncertainty quantification in (probabilistic) prediction and the notions we have in mind here, we consider a relatively simple example for illustration, one in which there are no covariates.  That is, suppose we have a sequence of real-valued observables $Y_1,Y_2,\ldots$ and, based on the observations $Y^n = y^n$, the goal is to predict $Y_{n+1}$ in a probabilistic way.  One standard way to approach this is to construct a Bayesian predictive distribution.  This requires specification of a prior distribution over the space of models, an updating step whereby the prior distribution is updated to posterior distribution via Bayes's theorem in light of the observation $Y^n=y^n$, and then that the posterior is converted into a predictive distribution for $Y_{n+1}$, which we will denote by $\Pi^n$; keep in mind that $\Pi^n$ is a function data $Y^n$.  

Of course, there are a number of different summaries one can extract from the predictive distribution $\Pi^n$.  Very often, the only summary considered is a prediction interval, e.g., the smallest set $A$ such that $\Pi^n(A)$ is no less than $1-\alpha$, for some specified level $\alpha \in (0,1)$.  Let this prediction interval be denoted by $C_\alpha(y^n)$.  By construction, $C_\alpha(y^n)$ has posterior predictive probability at least $1-\alpha$, but one typically wants to give this a frequentist interpretation, to conclude that $C_\alpha(Y^n)$ has prediction coverage probability at least the nominal level $1-\alpha$; see Equation~\eqref{eq:coverage} below. In many cases, the Bayesian prediction interval will satisfy this frequentist coverage probability property, at least approximately.  But one might ask: if the goal is to get a prediction interval that attains certain frequentist coverage properties, then why go to the trouble of constructing a full posterior predictive distribution for $Y_{n+1}$?  There are certain advantages to quantifying uncertainty with a full predictive distribution, so these deserve exploration.  

At a fundamental level, it is the predictive probabilities, i.e., $\Pi^n(A)$ for various $A$, that should be meaningful to the data analyst who opted to construct $\Pi^n$.  That is, values of $\Pi^n(A)$ that are large (resp.~small) should suggest that the data show strong (resp.~weak) support for the claim ``$Y_{n+1} \in A$.''  Therefore, based on the observed data and his predictive distribution construction, the data analyst would be inclined to conclude that the aforementioned claims will hold for $A$ with large $\Pi^n(A)$ and will not hold for those with small $\Pi^n(A)$.  But if the data analyst is thinking about the predictive distribution as a {\em method} for prediction, rather than summarizing his personal beliefs about $Y_{n+1}$, then he should care about the reliability of this method. This begs the following question: as a function of $Y^{n+1}$, do the two events $\{\text{$\Pi^n(A)$ is small}\}$ and $\{Y_{n+1} \not\in A\}$ tend to happen simultaneously for all the relevant $A$'s?  If not, then the predictive distribution, treated as a method for predictive inference, lacks reliability in the sense that there is risk of erroneous predictions.  Put differently, suppose the data analyst is a gambler who uses his $\Pi^n(A)$ values to set prices for \$1 bets on the uncertain outcomes ``$Y_{n+1} \in A$.'' Then a lack of reliability in the sense described above implies existence of some $A$ for which the gambler tends to assign a low price to ``$Y_{n+1} \in A$'' and have to pay out \$1.  Of course, a tendency to lose \$1 on low-priced bets can easily lead to ruin. The details in the above discussion will all be formalized in Section~\ref{S:validity}. 

Do the common Bayesian predictive distributions achieve this sort of reliability?  One way to assess this would be to consider the function 
\begin{equation}
\label{eq:f}
f(\alpha) = \prob\{\Pi^n(A) \leq \alpha, \, Y_{n+1} \in A\}, \quad \alpha \in (0,1),
\end{equation}
where $\prob$ denotes the joint distribution of $Y^{n+1}$. This function depends implicitly on $A$, on the chosen construction $y^n \mapsto \Pi^n$ of the predictive distribution, and on the underlying distribution $\prob$. It will be argued below that reliability or, rather, {\em validity} of a probabilistic predictor corresponds to 
\[ f(\alpha) \leq \alpha \quad \text{for all $(\alpha, n, A, \prob)$}. \]
Figure~\ref{fig:no.x} plots the function $f$ for $n=5$, $A=[3,5]$, and $\prob$ corresponding to iid $\unif(-5,5)$ random variables.  This is done for two Bayesian predictive distributions:
\begin{itemize}
\item a parametric version based on a simple iid normal model with a conjugate normal--inverse gamma prior, leading to a suitable Student-t predictive distribution; 
\vspace{-2mm}
\item and a nonparametric version based on the predictive distribution from a Dirichlet process mixture of normals model \citep[e.g.,][Ch.~5]{ghosh2003bayesian}; in fact, this is based on the (non-asymptotic) approximation in \citet{hahn.martin.walker.pred}. 
\end{itemize}
In both cases, we clearly see that there is an interval of $\alpha$ values at which the condition ``$f(\alpha) \leq \alpha$'' fails, hence the Bayesian predictive distribution is not valid in the sense described vaguely above, and more precisely in Section~\ref{S:validity}. That this is not specific to this example and these choices of predictive distribution is established in Corollary~\ref{prop:fct} below.  That precise predictive distribution can fail to be reliable in the sense above motivates our investigation into other probabilistic predictor constructions that are reliable, which necessarily must take the form of imprecise probabilities.  

\begin{figure}[t]
\begin{center}
\scalebox{0.7}{\includegraphics{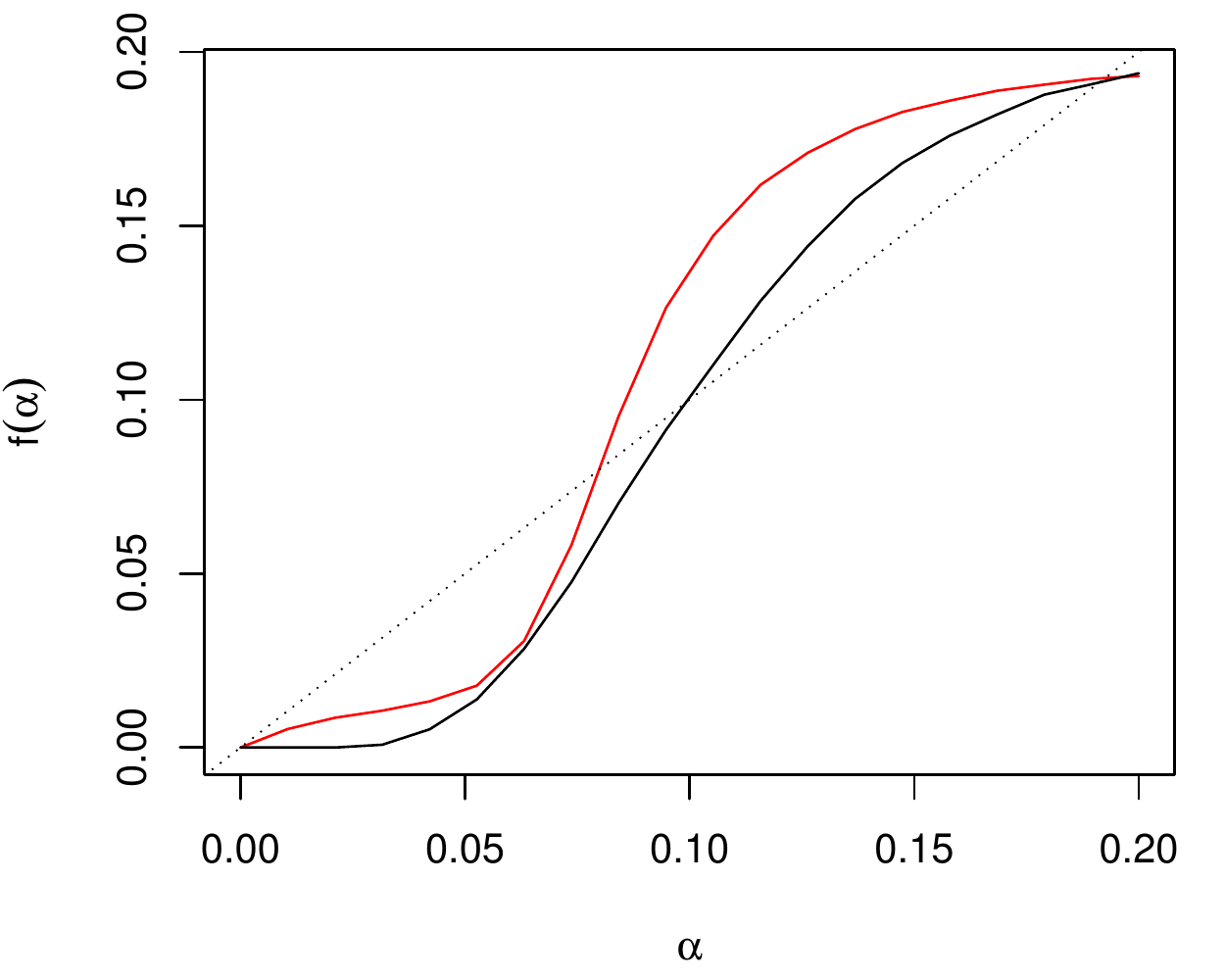}}
\end{center}
\caption{Plot of the function $f$ in \eqref{eq:f} corresponding to the two Bayesian predictive distributions (and other settings) as described in the text.}
\label{fig:no.x}
\end{figure}

Of course, Bayesian predictive distributions are not advertised to be reliable in this sense, so various excuses can be given, e.g., that the condition is too strong, that a different choice of prior distribution would perform better, etc.  To us, however, the fact that a Bayesian probabilistic predictor (or any other construction based on a precise probability distribution for that matter) is sure to put the data analyst and gambler at risk of systematically misleading conclusions and ruin, respectively, is a serious concern.  If precise probability were our only option, then of course we would just have to live with the aforementioned risk.  Here we show, however, that suitably incorporating imprecision into the construction can nullify these risks without sacrificing on the other desirable properties that probabilistic prediction affords.

\section{Prediction validity}
\label{S:validity}

\subsection{Setup}

The goal here is to formalize the ideas discussed in Section~\ref{s:ilust} above. Recall that the present paper is concerned with prediction in supervised learning problems, so we assume there is an exchangeable process $Z^{\infty} = (Z_1,Z_2,\ldots)$ with distribution $\prob$, where each $Z_i$ is a pair $(X_i,Y_i) \in \ZZ = \XX \times \YY$.  As is customary, ``$\prob(Z^n \in B)$'' is understood to mean the marginal probability for the event ``$Z^n \in B$'' derived from the joint distribution of $Z^\infty$ under $\prob$.  The distribution $\prob$ is completely unknown, except that it belongs to the user-defined model $\model$ consisting of exchangeable distributions. As is typical in statistical learning applications, we want to avoid strong model assumptions, which amounts to assuming $\model$ is large, i.e., is not indexed by a finite-dimensional parameter. Given the observed data $Z^n$ and a value $x_{n+1}$ of $X_{n+1}$, the goal is to reliably predict the corresponding $Y_{n+1}$.  As discussed in Sections~\ref{S:intro}--\ref{s:ilust}, a common strategy is to construct a {\em prediction set} aimed to achieve the nominal frequentist coverage probability. That is, a collection of functions $\plint_{n,\alpha}$, from $\ZZ^n \times \XX$ to subsets of $\YY$, indexed by $\alpha \in [0,1]$ and $n \geq 1$, defines a family of $100(1-\alpha)$\% prediction sets for $Y_{n+1}$ if 
\begin{equation}
\label{eq:coverage}
\inf_{\prob \in \model}\prob\{\plint_{n,\alpha}(Z^n,X_{n+1}) \ni Y_{n+1}\} \geq 1-\alpha, \quad \text{for all ($\alpha$, $n$)}.
\end{equation}
However, as discussed in Section~\ref{s:ilust}, a more ``complete'' uncertainty quantification about $Y_{n+1}$ may be desired, beyond prediction sets. 
To formalize this, we follow \citet{CellaMartinConformal} and define a {\em probabilistic predictor} as a map $(z^n,x) \mapsto (\lPi_x^n, \uPi_x^n)$, where $(\lPi_x^n, \uPi_x^n)$ is a pair of lower and upper predictive probabilities for the corresponding $Y_{n+1}$; for notational simplicity, the probabilistic predictor's dependence on the observed data $z^n$ is encoded in the superscript ``$n$'' only.  Then uncertainty quantification about $Y_{n+1}$, given $z^n$ and $X_{n+1}=x$, is provided by the function $A \mapsto (\lPi_x^n(A), \uPi_x^n(A))$.  

We are defining the probabilistic predictor for all $n$, but it could be that some minimum sample size is needed in order to properly define it.  For example, if some standardization procedure is being employed, then it would be necessary to have $n$ large enough to estimate standard errors.  As a rule in what follows, if $n$ is smaller than the necessary sample size, then we will silently take the probabilistic predictor to be vacuous, i.e., assign lower and upper probabilities 0 and 1, respectively, to every assertion.  

What kind of mathematical form should the function $A \mapsto (\lPi_x^n(A), \uPi_x^n(A))$ take?  Let $\A$ denote a $\sigma$-algebra of subsets of $\YY$ that are measurable with respect to the (common) marginal of the $Y_i$'s under $\prob$. We will assume that $\A$ is rich enough to contain the singletons, e.g., like the Borel $\sigma$-algebra.  Then the lower and upper probabilities are capacities defined on $\A$, i.e., monotone set functions, taking value 1 at $\YY$ and value 0 at $\varnothing$.  However, being ``lower'' and ``upper'' suggests a link between the two.  We formalize by requiring that, for each $z^n$ and new value $x$ of the feature $X_{n+1}$, the upper probability $\uPi_x^n$ for $Y_{n+1}$ is sub-additive; in particular, for any disjoint $A$ and $A'$, the upper probability satisfies $\uPi_x^n(A \cup A') \leq \uPi_x^n(A) + \uPi_x^n(A')$.  Then the lower probability $\lPi_x^n$ is defined as the dual or conjugate to the upper probability, 
\begin{equation}
\label{eq:dual}
\lPi_x^n(A) = 1 - \uPi_x^n(A^c), \quad A \in \A, 
\end{equation}
and from sub-additivity is follows that 
\[ \lPi_x^n(A) \leq \uPi_x^n(A), \quad A \in \A, \]
hence the name ``lower'' and ``upper'' probabilities.  Ordinary or precise probabilities are (sub)additive so they satisfy these condition with $\lPi_x^n \equiv \uPi_x^n$.  Moreover, all of the standard imprecise probability models---belief functions, possibility measures, lower/upper previsions---satisfy these conditions, so our assumptions corresponding to no loss of generality.  Since we will be interested in the statistical properties of the probabilistic predictor as functions of the data, we will assume that $(Z^n, X_{n+1}) \mapsto (\lPi_{X_{n+1}}^n(A), \uPi_{X_{n+1}}^n(A))$ is measurable for each $n \geq 1$ and for each $A \in \A$. 

The interpretation of the probabilistic predictor's output is subjective and goes as follows.  For given data $z^n$ and a new value $x$ of the feature $X_{n+1}$, the lower and upper probabilities represent 
\begin{align*}
\lPi_x^n(A) & = \text{maximum buying price for the gamble \$$1(Y_{n+1} \in A)$} \\
\uPi_x^n(A) & = \text{minimum selling price for the gamble \$$1(Y_{n+1} \in A)$},
\end{align*}
where $1(B)$ denotes the indicator of the event $B$. Therefore, based on data $z^n$ and new feature $x$, if the investigator's $\lPi_x^n(A)$ is large, then she would be inclined to buy the gamble \$$1(Y_{n+1} \in A)$, whereas, if her $\uPi_x^n(A)$ is small, then she would be inclined to sell the gamble \$$1(Y_{n+1} \in A)$; otherwise, she might choose to neither buy nor sell the gamble. For this reason, $\lPi_x^n(A)$ measures the subjective degree of belief and $\uPi_x^n(A)$ the plausibility of the event ``$Y_{n+1} \in A$.''  Below we introduce an element of objectivity through a requirement that its predictions be reliable in a statistical sense.

\subsection{Definition}

So far, we have imposed minimal mathematical constraints on the probabilistic predictor, plus its interpretation is subjective, so virtually no construction can be ruled out at this point.  However, the probabilistic predictor's practical utility requires that the uncertainty quantification derived from it be reliable in a certain sense.  The particular sense we have in mind is {\em statistical}.
That is, we require that inferences drawn based on the probabilistic predictor not be systematically misleading.  Based on the interpretations of the lower and upper probabilities described above, events of the general form 
\[ \{(z^n, x_{n+1}, y_{n+1}): \text{$\lPi_{x_{n+1}}^n(A)$ is large and $y_{n+1} \not\in A$}\} \]
and 
\[ \{(z^n, x_{n+1}, y_{n+1}): \text{$\uPi_{x_{n+1}}^n(A)$ is small and $y_{n+1} \in A$}\}, \]
should they occur, put the investigator at risk of making erroneous predictions and incurring losses, monetary or otherwise.  To protect the investigator from this risk, we impose the following condition on probabilistic predictors, ensuring that the aforementioned undesirable, risk-creating events are controllably rare.  




\begin{defn}
\label{def:valid}
The probabilistic predictor $(Z^n,x) \mapsto (\lPi_x^n, \uPi_x^n)$ is {\em valid} if one and, hence, both of the following equivalent conditions hold:
\begin{align}
\sup_{\prob \in \model}\prob\{ \lPi_{X_{n+1}}^n(A) \geq 1- \alpha \, , \, Y_{n+1} \notin A\} & \leq \alpha, \quad \text{for all $(\alpha, n, A)$} \label{eq:valid.lo} \\ 
\sup_{\prob \in \model}\prob\{ \uPi_{X_{n+1}}^n(A) \leq \alpha \, , \, Y_{n+1} \in A\} & \leq \alpha, \quad \text{for all $(\alpha, n, A)$}. \label{eq:valid.up}
\end{align}
Here ``for all $(\alpha, n, A)$'' is short for ``for all $\alpha \in [0,1]$, all $n \geq 1$ and all $A \in \A$.'' The two conditions are equivalent by the duality in \eqref{eq:dual} and the ``for all $A$'' clause. 
\end{defn}

The key point, again, is that validity ensures the probabilistic predictor will not tend to assign small upper probability to assertions about $Y_{n+1}$ that happen to be true, or large lower probability to assertions about $Y_{n+1}$ that happen to be false.  Practically, this ensures that the data analyst is not making systematically misleading predictions.  Such assurances are also fundamentally important to the logic of statistical reasoning.  Following \citet[][p.~42]{fisher1973}, what makes an observation leading to, say, a small value of $\uPi_{x_{n+1}}^n(A)$ informative about the claim ``$Y_{n+1} \in A$'' is that a logical disjunction is created: {\em either the claim does not hold or a small-probability event has occurred}.  Since small-probability events rarely occur, if we observe a small value of $\uPi_{x_{n+1}}^n(A)$, then we are inclined to conclude that $Y_{n+1} \not\in A$.  Validity also has a number of interesting and practically relevant consequences, which we explore in Section~\ref{S:implications}. 

Before moving on, we should mention some connections with certain notions of ``frequency calibration'' in the imprecise probability literature.  In particular, using our terminology and notation, \citet{denoeux2006} defines a probabilistic predictor to have a ``$100(1-\alpha)$\% confidence property,'' for a fixed $\alpha \in [0,1]$, if 
\[ \prob\bigl\{ \uPi_{X_{n+1}}^n(A) \geq \prob(Y_{n+1} \in A \mid Z^n, X_{n+1}) \, \text{ for all $A$} \bigr\} \geq 1-\alpha. \]
This and other variations are discussed more recently in \citet{denoeux.li.2018}.  Obviously, since the event on the left-hand side does not explicitly depend on $\alpha$, it must be that the probabilistic predictor depends implicitly on the specified $\alpha$ value, and various approaches to incorporate this $\alpha$-dependence so that the above property can be achieved are given in the aforementioned references.  The key observation is that calibration requires some relation between the probabilistic predictor for $Y_{n+1}$ and the true conditional distribution of $Y_{n+1}$.  In particular, the prediction upper probability ought to dominate the true conditional probability in some sense.  A similar dominance appears in our definition of validity, but a slight reformulation is needed.  Using iterated expectation, by conditioning on $(Z^n,X_{n+1})$, it is easy to see that \eqref{eq:valid.up} is equivalent to 
\begin{equation}
\label{eq:conditional}
\sup_{\prob \in \model} \E\bigl[ 1\{\uPi_{X_{n+1}}^n(A)\leq \alpha\} \, \prob(Y_{n+1} \in A \mid Z^n, X_{n+1}) \bigr] \leq \alpha, 
\end{equation}
where the expectation is with respect the same $\prob$ over which the supremum is taken.  That is, our notion of validity implies that, when restricted to data sets $(Z^n, X_{n+1})$ for which $\uPi_{X_{n+1}}^n(A)$ is small, the true conditional probability $\prob(Y_{n+1} \in A \mid Z^n, X_{n+1})$ cannot be any bigger on average. Incidentally, there are other notions of calibration/validity in the literature that concern matching up posited predictive distributions with the true probabilities in an average sense, not so unlike what \eqref{eq:conditional} achieves.  See, for example, the calibration property of Venn--Abers predictors in \citet{vovk.petej.2014} and the calibration safety property in \citet{grunwald.safe}.

\subsection{A stronger notion}
\label{s:strong}

That the calibration property imposed in \eqref{eq:valid.up} is required to hold for all $A \subseteq \YY$ might seem overly strong, but it turns out that there is an even stronger property that is particularly useful and can be readily attained.  To state this new property, however, we need some additional notation.  Define the probabilistic predictor's {\em plausibility contour} as 
\begin{equation}
\label{eq:first.contour}
\pi_x^n(y) = \uPi_x^n(\{y\}), \quad x \in \XX, \quad y \in \YY. 
\end{equation}
This is just the upper probability---or plausibility---assigned to singleton assertions about $Y_{n+1}$ of the form $A=\{y\}$, for generic $y \in \YY$.  In general, the plausibility contour is just one of the probabilistic predictor's many features.  But in the important special case where the probabilistic predictor has the mathematical property of {\em consonance}, the plausibility contour actually determines the entire probabilistic predictor.  We will discuss this latter point further below. 

\begin{defn}
\label{def:uvalid}
The probabilistic predictor $(Z^n,x) \mapsto (\lPi_x^n, \uPi_x^n)$ is {\em uniformly valid} if 
\begin{equation}
\label{eq:uvalid.alt}
\sup_{\prob \in \model}\prob\{ \pi_{X_{n+1}}^n(Y_{n+1}) \leq \alpha\} \leq \alpha, \quad \text{for all $(\alpha, n)$}, 
\end{equation}
\end{defn}

The condition \eqref{eq:uvalid.alt} is familiar, at least when connections are drawn to other contexts.  In particular, \eqref{eq:uvalid.alt} closely resembles the properties satisfied by p-values from hypothesis testing in classical statistics.  It is also effectively the same as the so-called {\em fundamental frequentist principle}, or {\em FFP}, in \citet{walley2002}.  But there are still some unanswered questions: in what sense is this definition stronger than that in Definition~\ref{def:valid}, and why do we call this ``uniform'' validity?  The following lemma helps us to answer both.

\begin{lem0}
Uniform validity in the sense of Definition~\ref{def:uvalid} is equivalent to the probabilistic predictor satisfying the following two properties:
\begin{align}
\sup_{\prob \in \model}\prob\{\text{$\lPi_{X_{n+1}}^n(A) \geq 1-\alpha$ and $Y_{n+1} \not\in A$ for some $A$}\} & \leq \alpha, \quad \text{for all $(\alpha, n)$} \label{eq:uvalid.lo} \\
\sup_{\prob \in \model}\prob\{\text{$\uPi_{X_{n+1}}^n(A) \leq \alpha$ and $Y_{n+1} \in A$ for some $A$}\} & \leq \alpha, \quad \text{for all $(\alpha, n)$}. \label{eq:uvalid.up} 
\end{align}
\end{lem0}

\begin{proof}
That both probabilities on the left-hand sides of \eqref{eq:uvalid.lo} and \eqref{eq:uvalid.up} are equal to the left-hand side of \eqref{eq:uvalid.alt} follows from the probabilistic predictor's monotonicity property.
\end{proof}

That uniform validity in the sense of Definition~\ref{def:uvalid} is stronger than validity in the sense of Definition~\ref{def:valid} can now be readily seen.  Indeed, the ``for some $A$'' inside the probability statement in \eqref{eq:uvalid.up} is effectively a union of $A$-dependent events like those in \eqref{eq:valid.up} over all $A$.  So if the union over $A$ of these $A$-dependent events has probability bounded by $\alpha$, then so would any individual event in that union.  This also explains our choice to describe this as ``uniform validity.''  That is, instead of a ``$\leq \alpha$'' bound that holds for each individual assertion $A$, it now must hold simultaneously or uniformly over all such $A$.  

This generalization is important for several reasons.  One of those reasons is technical; see Proposition~\ref{prop:coverage} below.  Another concerns the point that when the ``for all $A$'' clause on the outside of the probability statement in \eqref{eq:valid.up} is moved to the inside, the choice of $A$ can be data-dependent.  To see why this data-dependence might be relevant, consider a gambling scenario in which the agent's opponents have access to the data $(z^n,x)$ at the time of prediction. This allows the opponent to use the data to make strategic choices about which assertions $A$ to negotiate with the agent. Of course, if the agent's opponents can make these more sophisticated data-dependent plays while he is only able to control errors for assertions specified in advance, then that puts him at risk.  Uniform validity, however, protects the agent from this more subtle type of risk.  

Although we currently lack a formal proof, our experience suggests that only {\em consonant} \citep[][Ch.~10]{shafer1976mathematical} probabilistic predictors can achieve uniform validity.  The reason being that, if the plausibility contour, $\pi_x^n$, in \eqref{eq:first.contour}, is restricted in the sense that it cannot attain values arbitrarily close to 1, then the stochastically-no-smaller-than-uniform condition in \eqref{eq:uvalid.alt} likely cannot hold.  And it is precisely this arbitrarily-close-to-1 property that determines consonance; that is, a probabilistic predictor is consonant if and only if its plausibility contour satisfies 
\begin{equation}
\label{eq:sup}
\sup_y \pi_x^n(y) = 1, \quad \text{for all $(z^n, x)$}. 
\end{equation}
In this case, the probabilistic predictor takes the mathematical form of a possibility measure \citep{dubois.prade.book}, and is  determined by its contour function through the relationship 
\begin{equation}
\label{eq:cons}
\uPi_x^n(A) = \sup_{y \in A} \pi_x^n(y), \quad A \in \A. 
\end{equation}
Since the lower and upper probabilities being determined by a single point-function, as opposed to genuine set-functions, consonance amounts to a substantial simplification of the probabilistic predictor.  For our purposes here, and for statistical inference in general \citep{imprecisefrequentist}, this simplification comes with no loss of generality or flexibility.

\section{Implications of prediction validity}
\label{S:implications}

\subsection{Behavioral}

Despite our focus on frequentist-style properties, validity has some important behavioral consequences, \`a la de Finetti, Walley, and others.  
Towards this, define 
\[ \lgamma_n(A) = \inf_{(z^n, x) \in \ZZ^n \times \XX} \lPi_x^n(A) \quad \text{and} \quad \ugamma_n(A) = \sup_{(z^n, x) \in \ZZ^n \times \XX} \uPi_x^n(A), \]
the lower/upper probabilistic predictor evaluated at $A$, optimized over all of its data inputs; recall that $\lPi_x^n$ and $\uPi_x^n$ depend implicitly on an argument $z^n$.  An especially poor specification of prediction probabilities is a situation in which, for some $A \subseteq \YY$, 
\begin{equation}
\label{eq:sure.loss}
\lgamma_n(A) > \inf_{\prob \in \model}\prob(Y_{n+1} \in A) \quad \text{or} \quad \ugamma_n(A) < \sup_{\prob \in \model}\prob(Y_{n+1} \in A). 
\end{equation}
We will refer to this as (one-sided) {\em contraction}.  Ideally, the probabilistic predictor would mimic the true conditional probability at least in the sense that its average over data inputs would not be far from the true marginal probability.  So a situation like in \eqref{eq:sure.loss}, where the probabilistic predictor might be uniformly bounded away from the true marginal probability, is a sign of potential trouble.  For example, if your $\uPi_x^n(A)$ is smaller than the upper bound on the marginal probability of $A$, uniformly in $(z^n, x)$, i.e., {\em no matter what data is observed}, then arguably you should have had a tighter bound on your marginal probability to start.  Inconsistencies like this factor in to the behavioral properties of the probabilistic predictor, and (imprecise) probabilities more generally.  For example, the {\em sure loss} property---see Condition~(C7) in \citet[][Sec.~6.5.2]{walley1991} or Definition~3.3 in \citet{gong.meng.update}---corresponds to an extreme version of contraction where 
\[ \lgamma_n(A) > \sup_{\prob \in \model}\prob(Y_{n+1} \in A) \quad \text{or} \quad \ugamma_n(A) < \inf_{\prob \in \model}\prob(Y_{n+1} \in A). \]
In a gambling context, an inconsistency as severe as in the above display can be leveraged by your opponent to make you a sure loser.  We show below, in Proposition~\ref{prop:no.sure.loss}, that validity and one-sided contraction are incompatible; in particular, validity implies no sure loss.  

Although \eqref{eq:sure.loss} is still a rather strong condition, corresponding to a poor prediction probability specification, there are practically relevant cases where \eqref{eq:sure.loss} holds and creates a genuine risk.  We discuss this below following the proof. 

\begin{prop}
\label{prop:no.sure.loss}
Suppose that the probabilistic predictor, $(z^n,x) \mapsto (\lPi_x^n, \uPi_x^n)$, suffers from one-sided contraction in the sense that \eqref{eq:sure.loss} holds for some $A \subseteq \YY$. Then validity in the sense of Definition~\ref{def:valid} fails.  
\end{prop}

\begin{proof}
We present the argument here for the case where $\ugamma_n(A) < \sup_{\prob \in \model}\prob(Y_{n+1} \in A)$; the argument for the $\lgamma_n(A)$ bound is very similar.  For the assertion $A$ in \eqref{eq:sure.loss}, define 
\[ \xi_n(A,\alpha) = \sup_{\prob \in \model}\prob\{ \uPi_{X_{n+1}}^n(A) \leq \alpha \, , \, Y_{n+1} \in A\}, \]
so that \eqref{eq:valid.up} is equivalent to
\begin{equation}
\label{eq:strong.alt}
\xi_n(A,\alpha) \leq \alpha \quad \text{for all $(A,\alpha,n)$}. 
\end{equation}
By \eqref{eq:conditional}, we have 
\[ \xi_n(A,\alpha) = \sup_{\prob \in \model} \E\bigl[ 1\{\uPi_{X_{n+1}}^n(A)\leq \alpha\} \, \prob(Y_{n+1} \in A \mid Z^n, X_{n+1}) \bigr]. \]
Since $\uPi_{X_{n+1}}^n(A) \leq \ugamma_n(A)$ by definition, we get 
\[ 1\{\uPi_{X_{n+1}}^n(A)\leq \alpha\} \geq 1\{\ugamma_n(A) \leq \alpha\}. \]
From the alternative representation of $\xi_n(A,\alpha)$, since the lower bound in the above display is constant, it follows that 
\begin{equation}
\label{eq:xi.bound}
\xi_n(A,\alpha) \geq 1\{\ugamma_n(A) \leq \alpha\} \, \sup_{\prob \in \model}\prob(Y_{n+1} \in A). 
\end{equation}
According to \eqref{eq:sure.loss}, there exists an $A \subseteq \YY$ and a threshold $\alpha \in [0,1]$ such that 
\[ \ugamma_n(A) < \alpha < \sup_{\prob \in \model}\prob(Y_{n+1} \in A). \]
Then from \eqref{eq:xi.bound}, with this choice of $(A,\alpha)$, 
\[ \xi_n(A,\alpha) \geq \sup_{\prob \in \model}\prob(Y_{n+1} \in A) > \alpha. \]
Then \eqref{eq:strong.alt} and, hence, \eqref{eq:valid.up} fails, so the claim follows.  
\end{proof}


Our one-sided contraction \eqref{eq:sure.loss} also resembles the (C8)~portion of the {\em coherence} property in \citet[][Sec.~6.5.2]{walley1991}, but it is missing the (C9)~portion.  Therefore, it appears that validity is not enough to imply coherence as advocated for by Walley and others.  

As discussed above, a very common strategy is one in which the probabilistic predictor is an ordinary probability, i.e., $(z^n, x) \mapsto \Pi_x^n$, where $\Pi_x^n$ is a (precise) probability distribution on $\YY$.  The illustration presented in Section~\ref{s:ilust} suggests that validity might fail when the probabilistic predictor is a (precise) probability distribution.  Of course, if the true distribution is known, i.e., if $\model$ is a singleton, then setting $\Pi_n^x$ equal to the true conditional distribution would be valid; see \eqref{eq:conditional}.  However, when $\model$ is big, as assumed here, we cannot expect that a probability distribution can accommodate both the inherent variability in the response and the uncertainty about the underlying distribution.  So we should anticipate that imprecision is needed in order for predictions to be valid in the sense of Definition~\ref{def:valid}.  The discussion below formalizes this claim.

A precise probabilistic predictor cannot accommodate uncertainty about the model by assigning a range of probabilities for a given assertion.  As such, the probabilities $\Pi_n^x(A)$ will tend to be strictly between 0 and 1.  But if the model $\model$ is large, we fully expect the infimum and supremum of $\prob(Y_{n+1} \in A)$ over $\model$ to be 0 and 1, respectively.  Therefore, if there exists an assertion $A$ such that 
\begin{equation}
\label{eq:contraction}
\inf_{(z^n,x) \in \ZZ^n \times x} \Pi_x^n(A) > 0 \quad \text{or} \quad \sup_{(z^n,x) \in \ZZ^n \times x} \Pi_x^n(A) < 1, 
\end{equation}
with inequalities strict, then \eqref{eq:sure.loss} holds and, by Proposition~\ref{prop:no.sure.loss}, validity fails.  One situation in which the inequalities are strict for some $A$ is when
\begin{equation}
\label{eq:tight} 
\{\Pi_x^n: (z^n, x) \in \ZZ^n \times \XX\} \quad \text{is a tight collection of distributions}. 
\end{equation}
Roughly speaking, tightness prevents the collection of distributions from ``drifting off to infinity,'' keep at least some amount of probability mass to the interior of $\YY$. Tightness always holds for compact $\YY$, at least in all practical cases; for non-compact $\YY$, it would need to be verified case-by-case, using specific features of the map $(z^n, x) \mapsto \Pi_x^n$.  In any case, tightness leads to strict contraction \eqref{eq:contraction} for some $A$, which implies one-sided contraction, which implies validity fails.  

\begin{cor}
\label{prop:fct}
If the probabilistic predictor, $(z^n, x) \mapsto \Pi_x^n$, is a precise probability distribution that satisfies \eqref{eq:tight}, then it is not valid.
\end{cor}

\begin{proof}
A direct consequence of Proposition~\ref{prop:no.sure.loss}. 
\end{proof}

The above result establishes a version of the {\em false confidence theorem} \citep{Ryansatellite, MARTIN2019IJAR} in the context of prediction.  It says roughly the following: only probabilistic predictors that take the form of an imprecise probability can be valid in the sense of Definition~\ref{def:valid}.  We do not expect the tightness condition \eqref{eq:tight} is essential to the conclusion of Corollary~\ref{prop:fct}, but we currently are not aware of a direct proof.

\subsection{Statistical}
\label{ss:tests}

Here we consider some more classical frequentist-style prediction tasks.  First, consider testing certain ``hypotheses'' about $Y_{n+1}$. For example, an investor may want to sell a certain asset when its price exceeds some fixed level, say $y^\star$.  So he would like to assess the plausibility of an assertion or hypothesis of the form ``$Y_{n+1} \in A$,'' for $A=[0,y^\star]$ and, in particular, decide if the new price being below the $y^\star$ threshold is too plausible to warrant taking quick action to sell.  We show below that the test 
\begin{equation}
\label{eq:test}
\text{reject ``$Y_{n+1} \in A$'' if and only if $\uPi_x^n(A) \leq \alpha$}, 
\end{equation}
derived from a valid probabilistic predictor controls the error probability at level $\alpha$.  

A more common prediction-related task is the construction of a prediction set, i.e., a set of sufficiently plausible values for $Y_{n+1}$ given the observed data.  A natural way to construct a  prediction set from a probabilistic predictor is 
\begin{equation}
\label{eq:pred.set}
\plint_{n,\alpha}(z^n, x) = \{y: \pi_x^n(y) > \alpha\}, 
\end{equation}
where $\pi_x^n$ is the plausibility contour \eqref{eq:first.contour} based on $(z^n, x)$. Compare this to a Bayesian highest predictive density region.  The following proposition shows that uniform validity implies that this is a genuine $100(1-\alpha)$\% prediction set in the sense that its frequentist coverage probability is at least the advertise/nominal level $1-\alpha$. 

\begin{prop}
\label{prop:coverage}
{\em (a)} If the probabilistic predictor is valid in the sense of Definition~\ref{def:valid}, then the test described in \eqref{eq:test} controls error rates at level $\alpha$ in the sense that 
\[ \sup_{\prob \in \model}\prob\{\text{test based on $(Z^n,X_{n+1})$ rejects and $Y_{n+1} \in A$}\} \leq \alpha. \]
{\em (b)} If the probabilistic predictor is uniformly valid in the sense of Definition~\ref{def:uvalid}, then \eqref{eq:pred.set} defines a genuine $100(1-\alpha)$\% prediction set in the sense that 
\[ \sup_{\prob \in \model}\prob\{ \plint_{n,\alpha}(Z^n, X_{n+1}) \not\ni Y_{n+1} \} \leq \alpha, \quad \text{for all $(\alpha, n)$}. \]
\end{prop}

\begin{proof}
Part~(a) is an immediate consequence of the definition of validity, in particular, \eqref{eq:valid.up}.  Part~(b) follows directly from \eqref{eq:uvalid.alt}. 
\end{proof}

Recall our conjecture that uniform validity is satisfied only for probabilistic predictors that are consonant, i.e., fully determined by their plausibility contour via \eqref{eq:cons}.  For consonant probabilistic predictors, the level sets of the plausibility contour, which are nested by definition, play an important role.  In particular, this underlying nested structure allows us to re-express the prediction set in \eqref{eq:pred.set} in terms of the lower probability:
\[ \plint_{n,\alpha}(z^n, x) = \textstyle\bigcap \{A: \lPi_x^n(A) \geq 1-\alpha\}. \]
That is, $\plint_{n,\alpha}(z^n, x)$ can also be interpreted as the smallest assertion $A$ about $Y_{n+1}$ to which the probabilistic predictor assigns lower probability at least $1-\alpha$.


\section{Inferential models}
\label{S:IM}

A relevant question is how to construct a probabilistic predictor that achieves the (uniform) validity condition.  One strategy would be through a {\em generalized Bayes} approach as advocated for in, e.g., \citet[][Sec.~6.4]{walley1991}.  That is, if $\model$ is the set of candidate joint distributions for the observables, the generalized Bayes rule would define an upper prediction probability as 
\begin{equation}
\label{eq:gbayes}
\uPi_x^n(A) = \sup_{\prob \in \model} \prob(Y_{n+1} \in A \mid Z^n, X_{n+1}=x), \quad A \in \A, 
\end{equation}
and corresponding lower probability by replacing $\sup$ by $\inf$.  That this satisfies validity in the sense of Definition~\ref{def:valid} follows from the alternative formulation in \eqref{eq:conditional}.  While this solution might have some appeal, there are also some reasons to be concerned.  First, this would tend to be quite conservative, i.e., a large model space $\model$ implies a wide gap between lower and upper probabilities.  Second, since generalized Bayes would not be consonant, it is doubtful that the desirable uniform validity in Definition~\ref{def:uvalid} holds.  So it is worth considering alternative constructions that might be more efficient.  

An inferential model (IM) is a data-dependent probabilistic structure designed to quantify uncertainty about unknowns, like the probabilistic predictor described above.  The difference is that, as the name suggests, IMs have traditionally focused on the {\em statistical inference} problem, where the unknowns are fixed quantities.  IMs have connections to various other approaches to statistical inference, some that quantify uncertainties with ordinary probabilities, e.g., Bayesian inference, fiducial inference \citep{fisherfiducial}, and generalized fiducial inference \citep{MainHaning}, and others with imprecise probabilities, e.g., Dempster--Shafer theory \citep{dempster.copss, dempster1967, dempster1968a, DEMPSTER2008365, shafer1976mathematical} and other belief function frameworks \citep{denoeux.li.2018, denoeux2014}. While there are some technical differences resulting from the unknown being fixed in the inference case and random in the prediction case, the common goal of providing valid uncertainty quantification is more or less the same.  Therefore, we expect that the key ideas behind the construction of a valid IM for inference ought to be applicable to the prediction problem as well, modulo a few adjustments. Below we describe a  construction of a probabilistic predictor that is valid in the sense described in Section~\ref{S:validity}. 

The general IM construction is composed of three steps. The A-step {\em associates} the observable data and unknown  quantity of interest with an unobservable auxiliary variable whose distribution is fully known. In the early work on IMs, this association was usually a complete description of the data-generating process.  For example, suppose we have, say, $n$ independent and identically distributed (iid) observations $Z_1,\ldots,Z_n$, collected into the vector $Z^n$, from a statistical model with unknown parameter $\theta$.  Then an association would effectively be a description of how to generate data $Z^n$ from that model, i.e., 
\[ Z^n = a(\theta, U^n), \]
where $U^n$ would typically be a vector of iid latent/auxiliary variables with a known distribution, e.g., $\unif(0,1)$.  While such an association can always be written down, there are a few obstacles one might face when trying to complete the IM construction: 
\begin{itemize}
\item When the dimension of $U^n$ is greater than that of $\theta$, as is typical, a dimension reduction step is recommended \citep{condmartin}, but this can be nontrivial. 
\item The association itself requires (more than) a fully specified statistical model for data, which may not be available in the application at hand. 
\end{itemize}
However, \citet{martin2015,MARTIN2018105} showed that the A-step's requirements can be relaxed.  All that is needed is an association that relates a  function of both the data and the unknowns to an unobservable auxiliary variable.  This idea has proved to be useful in a variety of classical \citep{cahoon2019generalized1,CAHOON202151} and modern \citep{CellaMartinBelief} inference problems, and here we develop a version suitable for prediction. 


Once a generalized association has been set, the remaining steps of the (generalized) IM construction proceed exactly as described in, say, \citet{mainMartin}.  Roughly, the P-step introduces a random set that aims to {\em predict} or guess the unobserved value of the auxiliary variable.  Easy to arrange properties of this user-specified random set ensure that the guessing of the auxiliary variable is done in a reliable way, which turns out to be fundamental for validity.  Next, the C-step {\em combines} the results of the A- and P-steps, yielding a new, data-dependent random set on the space where the quantity of interest resides.  Finally, this random set's distribution determines lower and upper probabilities that can be used to assign degrees of belief and plausibility to any relevant assertion about the unknown quantities of interest.  Below we describe the generalize IM construction in more detail for the prediction problem at hand. 

For prediction, the unknown is $Y_{n+1}$, not a model parameter as in the formulations described above.  So the kind of association needed is one that identifies a function of $(Z^n, Z_{n+1})$ that has a known distribution.  Once found, the three-step (generalized) IM construction proceeds as follows. 


\begin{astep}
Suppose there exists a function $\phi_n: \ZZ^n \times \ZZ \to \RR$ such that the distribution, say, $\Q_n$, of the random variable $\phi_n(Z^n, Z_{n+1})$ is known, i.e., does not depend on the unknown $\prob$.  Then associate the observable data $Z^n$ and the yet-to-be-observed $Z_{n+1}$ with the unobservable auxiliary variable $U$ as follows:
\begin{equation}
\label{eq:astep}
\phi_n(Z^n, Z_{n+1}) = U, \quad U \sim \Q_n.
\end{equation}
For our case where $Z_{n+1} = (X_{n+1}, Y_{n+1})$ and interest is in $Y_{n+1}$ for a given $X_{n+1}=x$, the association defines a set-valued mapping 
\[ (Z^n, x, u) \mapsto \YY_x^n(u) := \{y \in \YY: \phi_n(Z^n, (x,y)) = u\}. \]
\end{astep}

\begin{pstep}
Define a nested random set $\U$ (see below) on the space $\UU$ of the auxiliary variable $U$, designed to reliably contain realizations of $U \sim \Q_n$ in the sense of \eqref{eq:prs.valid} below.  The distribution of the random set $\U$ will be denoted by $\R_n$.   
\end{pstep}

\begin{cstep}
Combine the results of the A- and P-steps to get the data-dependent random set 
\[ \YY_x^n(\U) = \bigcup_{u \in \U} \YY_x^n(u) = \{y \in \YY: \phi_n(Z^n, (x,y)) \in \U\}. \]
Then the distribution of this new random set, derived from the distribution of $\U$, determines the probabilistic predictor for $Y_{n+1}$, i.e., 
\begin{equation}
\label{eq:im.output}
\begin{split}
\lPi_x^n(A) & = \R_n\{\YY_x^n(\U) \subseteq A\} \\
\uPi_x^n(A) & = \R_n\{\YY_x^n(\U) \cap A \neq \varnothing\}.
\end{split}
\end{equation}
\end{cstep}

\begin{remark}
\label{re:empty}
If $\YY_x^n(\U)$ is empty with positive $\R_n$-probability, then some adjustment to the probabilistic predictor in \eqref{eq:im.output} is needed.  This will be relevant for the classification problem in  Section~\ref{S:Classification}.  
\end{remark}


The above construction is abstract for the purpose of generality. The challenge is in identifying the function $\phi_n$, and examples will be given in Sections~\ref{S:Regression}--\ref{S:Classification} below.  Other examples were explored previously in \citet{Martin2016PriorFreePP} where $\prob$ was assumed to belong to a parametric family. 
Here, however, $\prob$ is not indexed by a finite-dimensional parameter, so different techniques are required.  The remainder of this section investigates the properties of the abstract probabilistic predictor construction above.  

The random set $\U$ is assumed to be nested in the sense that, for any two sets in its support, one is a subset of the other.  As a consequence, the derived probabilistic predictor is consonant; that is, its contour function, which is given by 
\begin{equation}
\label{eq:contour}
\pi_x^n(y) = \R_n\{\YY_x^n(\U) \ni y\}, \quad y \in \YY, 
\end{equation}
satisfies \eqref{eq:sup} and, hence, determines the entire probabilistic predictor through the relationship \eqref{eq:cons}, i.e., $\uPi_x^n(A) = \sup_{y \in A} \pi_x^n(y)$.  As discussed in Section~\ref{s:strong}, consonance is important---perhaps necessary---for uniform validity. 

It remains to establish that the probabilistic predictor resulting from the above construction is (uniformly) valid.  This requires stating the  conditions on $\U$ more precisely.  Since the cases in the following sections involve an auxiliary variable $U$ that is discrete, we will focus on the discrete case.  
First, define the random set's contour function 
\[ f(u) = \R_n(\U \ni u), \quad u \in \UU. \]
Then the required link between $\Q_n$ and $\R_n$ is that 
\begin{equation}
\label{eq:prs.valid}
\text{if $U \sim \Q_n$, then $f(U) \sim \unif((n+1)^{-1}\I_{n+1})$},
\end{equation}
where $\I_{n+1} = \{1,\ldots,n,n+1\}$, so that this uniform distribution is discrete.  With this link between the auxiliary variable's distribution $\Q_n$ and the random set's distribution $\R_n$, we are ready to state and prove the main result.  

\begin{thm}
\label{thm:valid}
If the random set $\U$ satisfies \eqref{eq:prs.valid}, and if $\YY_{X_{n+1}}^n(\U)$ is non-empty with $\R_n$-probability~1 for $\prob$-almost all $(Z^n,X_{n+1})$, then the probabilistic predictor defined in \eqref{eq:im.output}, or equivalently \eqref{eq:cons}, is uniformly valid in the sense of Definition~\ref{def:uvalid}.  
\end{thm}

\begin{proof}
First, for $Z^n$ and $Z_{n+1}=(X_{n+1}, Y_{n+1})$, set $U = \phi_n(Z^n, Z_{n+1})$.  Then  
\[ \YY_{X_{n+1}}^n(\U) \ni Y_{n+1} \iff \U \ni U. \]
The $\R_n$-probability of the left- and right-hand side events are $\pi_{X_{n+1}}^n(Y_{n+1})$ and $f(U)$, respectively, so these two random variables---the first as a function of $(Z^n,Z_{n+1}) \sim \prob$ and the second as a function of $U \sim \Q_n$---have the same distribution.  Equation \eqref{eq:prs.valid} states that $f(U)$ is uniform and, therefore, so is $\pi_{X_{n+1}}^n(Y_{n+1})$. 
\end{proof}

The non-emptiness condition is not necessary for validity, but some adjustment is needed to the definition in \eqref{eq:im.output}, as mentioned in Remark~\ref{re:empty}, to address this.  We will discuss this below in the specific application to classification in Section~\ref{S:Classification}. The requirement in \eqref{eq:prs.valid} that $f(U) \sim \unif((n+1)^{-1}\I_{n+1})$ can be relaxed in a certain sense without compromising validity. That is, validity also holds for any random set such that $f(U)$ is stochastically no smaller than $\unif((n+1)^{-1}\I_{n+1})$. However,  \citet{CellaMartinConformal} show that the choice of $\U$ whose corresponding $f(U)$ is exactly uniformly distributed is most efficient. Fortunately, this is easy to arrange; see \eqref{eq:S} below.

The following is an immediate consequence of the uniform validity conclusion above and the general results in Propositions~\ref{prop:no.sure.loss}--\ref{prop:coverage} in the previous section.  

\begin{cor}
Under the conditions of Theorem~\ref{thm:valid}, the probabilistic predictor defined in \eqref{eq:im.output} avoids sure loss in the sense of \eqref{eq:sure.loss} and admits a prediction set $C_{n,\alpha}$ as in \eqref{eq:pred.set} that achieves the nominal frequentist prediction coverage probability.
\end{cor}

Consequently, the proposed probabilistic predictor construction achieves the desired subjective/behaviorist and objective/frequentist properties simultaneously.  Two specific and practically relevant applications of this construction in the context of regression and classification will be presented in Section~\ref{S:Regression} and \ref{S:Classification}, respectively.  

It is important to point out that the kind of validity being considered here is {\em marginal}, which is easiest to understand in the context of calibrated prediction sets as in \eqref{eq:coverage}. That is, the conditional coverage probability of the prediction set is 
\[ x_{n+1} \mapsto \prob\{ \plint_\alpha^n(x_{n+1}) \ni Y_{n+1} \mid X_{n+1} = x_{n+1}\}, \]
a function of $x_{n+1}$.  Then the validity property implies that the expected value of this function, with respect to the marginal distribution of $X_{n+1}$ under $\prob$, is at least $1-\alpha$.  This marginal coverage guarantee, of course, says nothing about the conditional coverage at any particular $x_{n+1}$ values.  Conditional validity is both challenging and practically relevant, and we discuss this briefly in Section~\ref{S:discuss}.

\section{Probabilistic prediction in regression}
\label{S:Regression}

Recall that the A-step requires the specification of a real-valued function $\phi_n$, such that the distribution of $\phi_n(Z^n, Z_{n+1})$ is known. Towards this, given $Z^{n+1}= (Z^n,Z_{n+1})$ consisting of the observable $(Z^n, X_{n+1})$ and  the  yet-to-be-observed $Y_{n+1}$, consider first a transformation $Z^{n+1} \rightarrow T^{n+1}$, defined by
\begin{equation}
\label{eq:nonconform}
T_i = \Psi(Z^{n+1}_{-i},Z_i), \quad i \in \I_{n+1},
\end{equation}
where $Z_{-i}^{n+1} = Z^{n+1}\setminus \{(Y_i,X_i)\}$, and
$\Psi$ is a suitable real-valued function that compares $Y_i$ to a prediction derived from $Z_{-i}^{n+1}$ at $X_i$, being small if they agree and large if they disagree.  For example, to each $Z_{-i}^{n+1}$, one could fit a regression model to get an estimated mean response $\hat\mu_{-i}^{n+1}(X_i)$ and take $T_i$ as the corresponding absolute residual 
\begin{equation}
\label{eq:nonconform_reg}
 T_i = \bigl| Y_i - \hat\mu_{-i}^{n+1}(X_i) \bigr|, \quad i \in \I_{n+1}. 
\end{equation}
The critical property of $\Psi$ is that it be symmetric in the elements of its first vector argument. This symmetry guarantees that the assumed exchangeability in $Z_1,Z_2,\ldots$ is preserved when $Z^{n+1}$ get mapped to $T^{n+1}$. As $T_i$ depends on the entire data $Z^{n+1}$, we will write $T_i(Z^{n+1})$ where necessary to highlight that dependence. In regression, where the $Y_i$'s are continuous and $\Psi$ is non-constant on sets of $Y^{n+1}$ with positive $\prob$-probability, like the one in \eqref{eq:nonconform_reg}, so that there are no ties, a well-known consequence of exchangeability of $T_1,\ldots,T_{n+1}$ is that their ranks are marginally distributed according to $\unif(\I_{n+1})$, the discrete uniform law on $\I_{n+1}$.

Having identified a function of $(Z^n,Z_{n+1})$ whose distribution is known, we can complete the A-step of the IM construction by writing a version of \eqref{eq:astep} as follows:
\begin{equation}
\label{eq:astepCPrank}
r(T_{n+1}) = U, \quad U \sim \unif(\I_{n+1}),
\end{equation}
where $r(\cdot)$ is the ascending ranking operator. The choice of $T_{n+1}$ instead of any of the other $T_i$'s in \eqref{eq:astepCPrank} is simply because $T_{n+1}$ is the one that holds the to-be-predicted value, $Y_{n+1}$, in special
status.  Note that, while it appears this expression only depends on $T_{n+1}$, it does implicitly depend on all the $T_i$'s and, hence, all of $Z^{n+1}$, through the ranking procedure.  In summary, to complete the A-step, the only task for the data analyst is the specification of $\Psi$. It is worth to mention that, while validity of the probabilistic predictor is guaranteed for any suitable $\Psi$, choices of $\Psi$ that fail to capture the structure of the problem at hand can lead to inefficiency. 

For the P-step, the specification of a nested random set targeting the unobserved realization of the auxiliary variable $U$, introduced above, is needed. Consider 
\begin{equation}
\label{eq:S}
\U = \{1,2,\ldots,U'\}, \quad U' \sim \unif(\I_{n+1}).
\end{equation}
It is straightforward to show that this random set satisfies the critical calibration property \eqref{eq:prs.valid}.
Moreover, this choice also makes intuitive sense, as $\U$ always includes the value 1. This is desirable given the ascending ranking operator in \eqref{eq:astepCPrank} because it implies values of $Y_{n+1}$ that make the residual $T_{n+1}$ small will be assigned high plausibility. 

Finally, in the C-step, $\U$ is combined with the $u$-indexed collection of sets
\[\YY_{x_{n+1}}^n(u) = \bigl\{y_{n+1}: r\bigl(T_{n+1}(z^{n+1})\bigr) = u \bigr\} \]
that arise from the association \eqref{eq:astepCPrank}.  Here and below, note that $z^{n+1}$ consists of the observed $z^n$ values with $z_{n+1}=(x_{n+1},y_{n+1})$ appended to it.  The particular combination, as described in the previous section, 
It is easy to see that $\YY_{x_{n+1}}^n(\U)$'s corresponding contour function for $Y_{n+1}$ is given by
 \begin{align}
 \pi_{x_{n+1}}^n(y_{n+1}) & = \R_n\{\YY_{x_{n+1}}^n(\U) \ni y_{n+1}\} \nonumber \\
 & = \text{prob}\{\unif(\I_{n+1}) \geq r(T_{n+1}(z^{n+1}))\} \nonumber \\
& = \frac{1}{n+1} \sum_{i=1}^{n+1} 1\{T_i(z^{n+1}) \geq T_{n+1}(z^{n+1})\}. \label{eq:contourfunc}
\end{align}
As $\YY_{x_{n+1}}^n(\U)$ is both nested and non-empty, its contour function above is all that is needed to define a probabilistic predictor and, consequently, quantify uncertainty about any assertion $A \subset \YY$ of interest.  Uniform validity of this probabilistic predictor follows directly from the general result in Theorem~\ref{thm:valid}.  





For illustration, consider the following example. Let $X_1,\ldots,X_n$ be iid $\unif(0,1)$, with $n=200$, and let $Y_1,\ldots,Y_n$ be independent, where $Y_i = \mu(X_i) + 0.1 \varepsilon_i$, where $\mu(x) = \sin^3(2\pi x^{3})$, and $\varepsilon_1,\ldots,\varepsilon_n$ are iid from a Student-t distribution with $\text{df}=5$.  Figure~\ref{fig:covariates} displays the data, the true regression function $\mu(x)$ and the fitted regression curve $\hat{\mu}(x)$ based on a B-spline with 12 degrees of freedom. A 95\% prediction band is also displayed, derived by \eqref{eq:pred.set} and $x_{n+1}$ taking values along the observed $x^n$.

\begin{figure}[t]
\begin{center}
\subfigure[]{\scalebox{0.57}{\includegraphics{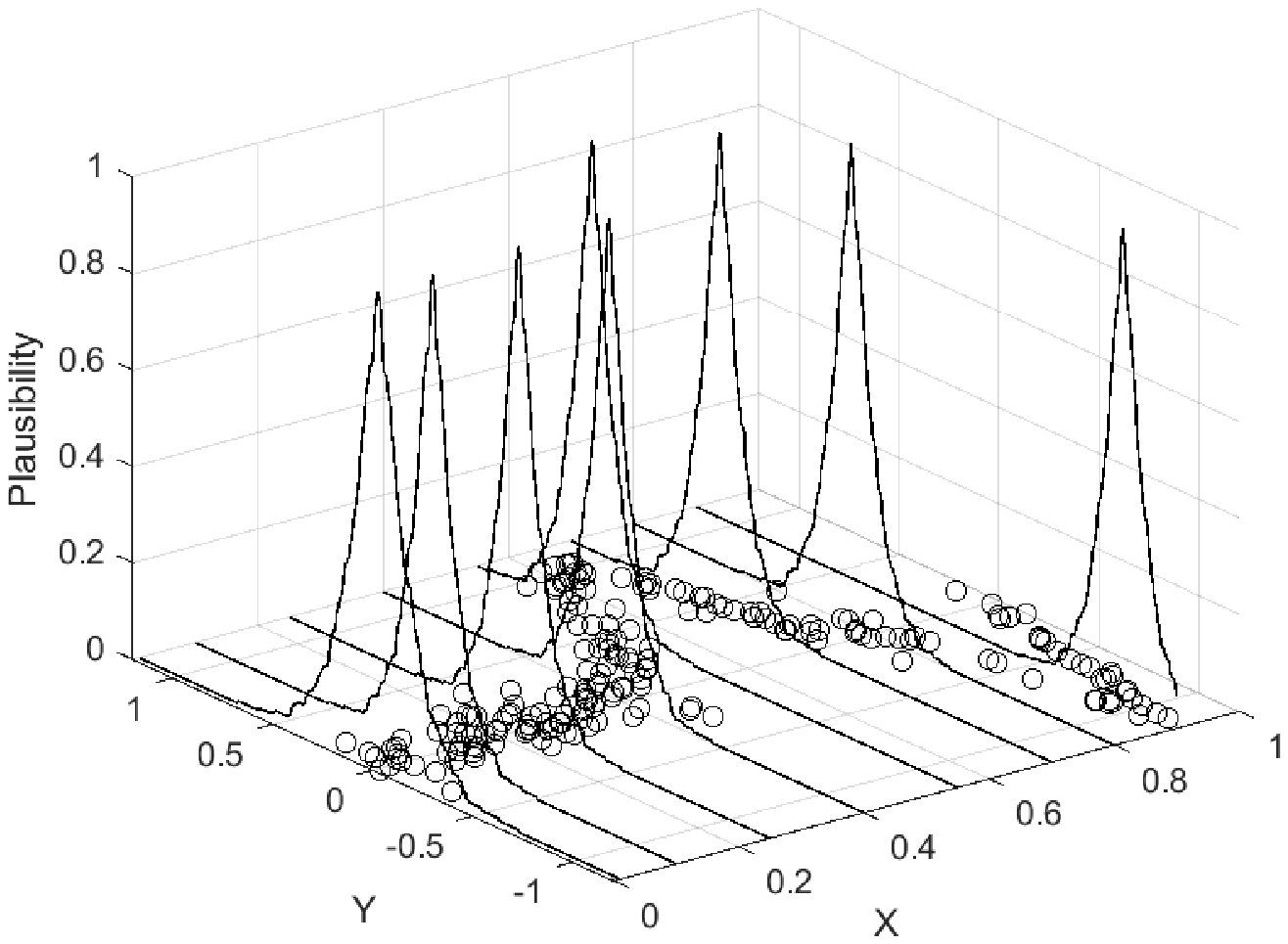}}}
\subfigure[]{\scalebox{0.6}{\includegraphics{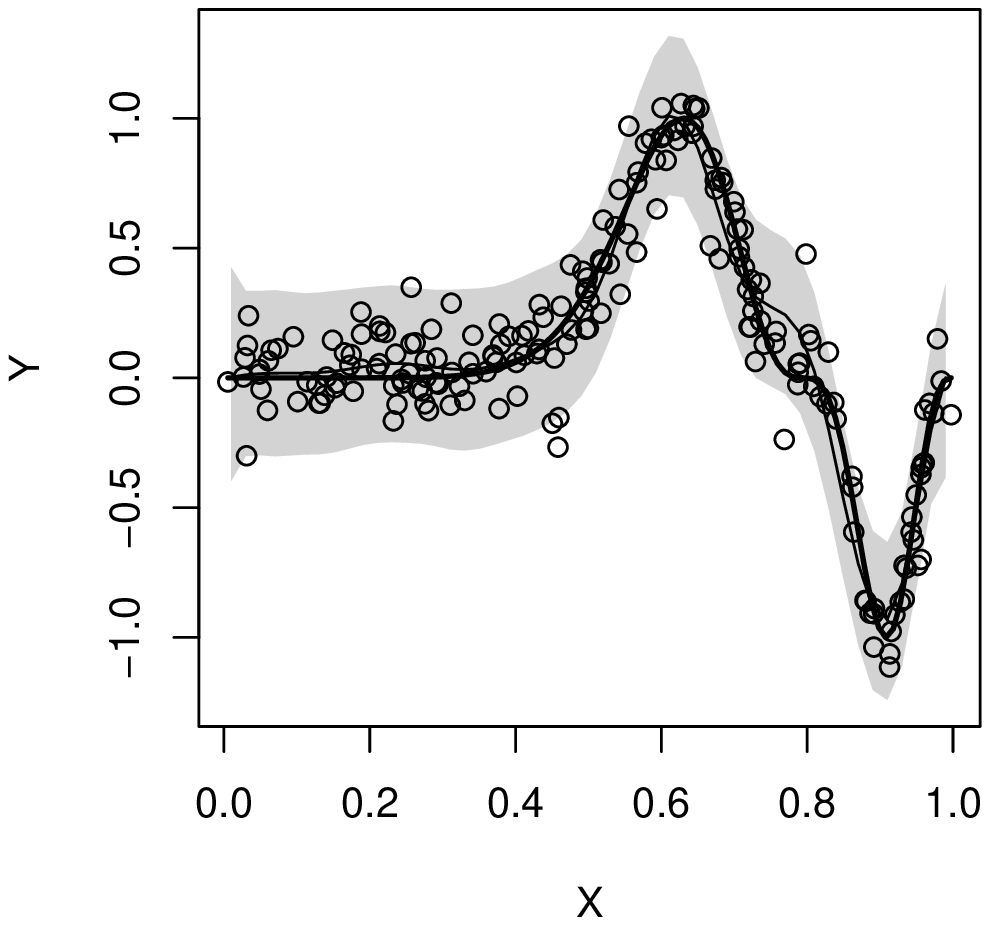}}}
\end{center}
\caption{Panel~(a): Data and the plausibility contours at selected values of $x$. Panel~(b): Data, the true mean curve (heavy line), the fitted B-spline regression curve (thin line), and the 95\% pointwise prediction band.}
\label{fig:covariates}
\end{figure}



We end this section pointing out an important connection between the prediction IM developed here and the powerful {\em conformal prediction} presented in \citet{Vovk:2005}.  The reader may have recognized the $\Psi$ function in the A-step of our construction as the so-called {\em non-conformity measure}, an essential component in the conformal prediction framework.  Moreover, the basic output from the IM construction presented below is the plausibility contour in \eqref{eq:contourfunc}, which is precisely conformal prediction's p-value or transducer. The theory in \citet{Vovk:2005} takes this conformal transducer, which is uniformly distributed as stated in Theorem~\ref{thm:valid}, and constructs a prediction set as in \eqref{eq:pred.set} with the prediction coverage probability property as in \eqref{eq:coverage}. It was recently recognized \citep{CellaMartinConformal} that the conformal prediction output could be converted into a uniformly valid probabilistic predictor in the sense of Definition~\ref{def:uvalid}, one that can make valid belief assignments, by treating the transducer as the contour of a consonant plausibility function via \eqref{eq:cons}.  We refer to this general probabilistic predictor construction as ``conformal + consonance,'' and all it requires is that the conformal transducer $\pi_x^n$ be a plausibility contour function in the sense that it satisfy $\sup_y \pi_x^n(y)=1$ for all $(z^n,x)$. This is easy to verify in cases where $Y$ is a continuous random variable.  Indeed, for the $\Psi$ function in \eqref{eq:nonconform_reg}, the supremum is attained at $y = \hat\mu_{-(n+1)}^{n+1}(x)$.  In other cases, like in classification where $Y$ is discrete, the ``conformal + consonance'' construction is not so straightforward.  We discuss these considerations next in Section~\ref{S:Classification}.

\section{Probabilistic prediction in classification}
\label{S:Classification}


In Section~\ref{S:Regression}, we found that the A-step boils down to the specification of a suitable real-valued, exchangeability-preserving function $\Psi$, which \citet{Vovk:2005} refer as a non-conformity measure. 
In binary classification problems, a $\Psi$ function like in \eqref{eq:nonconform_reg} can also be used here by encoding the binary labels as distinct real numbers. However, if there are more than two labels, and not in an ordinal scale where the assignment of different numbers to them is justified, there is no natural way to measure the distance between labels. Consequently, we cannot measure how wrong a prediction is---it is simply right or wrong \citep{Shafer2007ATO}. To circumvent this, \citet{Vovk:2005} suggest the following non-conformity measure based on nearest-neighbor classification:
\begin{equation}
\label{eq:nonconform_class}
\Psi(Z^{n+1}_{-i},Z_i) = \frac{\min_{j \in \I_{n+1}\setminus\{i\}:Y_j = Y_i}d(X_j,X_i)}{\min_{j \in \I_{n+1}\setminus\{i\}:Y_j \neq Y_i}d(X_j,X_i)},
\end{equation}
where $d$ is the Euclidean distance. In words, $\Psi(Z^{n+1}_{-i},Z_i)$ is large if $X_i$ is close to an element in $X^{n+1}_{-i}$ with a label different from $Y_i$ and far from any element in $X^{n+1}_{-i}$ with label equal to $Y_i$. If both the numerator and the denominator in \eqref{eq:nonconform_class} are 0, \citet{Shafer2007ATO} recommend taking the ratio also to be 0. Other non-conformity measures for classification problems can be found in \citet{Vovk:2005}.

Two factors were fundamental to the specification of the association \eqref{eq:astepCPrank} in Section~\ref{S:Regression}, namely the identification of $\Psi$, so that $Z^{n+1}$ can be mapped to $T^{n+1}$ preserving exchangeability, and the continuity of the $T_i$'s. In classification, however, the $Y_i$'s are not continuous, so there could be ties in the $T_i$'s.  Consequently, their ranks would be no longer uniform distributed on $\I_{n+1}$.  Luckily, when ties are possible, $r(T_{n+1})$ is stochastically no larger than the discrete uniform distribution it would take if there were no ties.  This leads to an ``association'' of the form 
\[ r(T_{n+1}) = U, \quad U \leq_{\text{st}} \unif(\I_{n+1}). \]
But for situations like this where the association involves a stochastic inequality, the general arguments in \citet[][Sec.~5]{marginalmartin} imply that the inequality can be ignored and the association \eqref{eq:astepCPrank}---with stochastic equality---can still be used.

Having identified the appropriate association, the IM construction proceeds analogously to that in the previous section: the A-step is completed by writing \eqref{eq:astepCPrank}, the random set \eqref{eq:S} is chosen in the P-step to target the unobserved realization of the auxiliary variable $U$, and, in the C-step, the ingredients in the A- and P-steps are combined to get $\YY_{x_{n+1}}^n(\U)$, a data-dependent random subset of $\YY$. However, due to the discreteness of $\YY$, it is possible that $\YY_{x_{n+1}}^n(\U)$ is empty with positive $\R_n$-probability. As discussed in Section~\ref{S:IM}, in these cases, some adjustment to the probabilistic predictor in \eqref{eq:im.output} is necessary to avoid the counter-intuitive ``conflict'' cases where realizations of the random set $\YY_{x_{n+1}}^n(\U)$ happens to be empty. There is a sense in which empty prediction sets could be meaningful, but we defer this discussion to Section~\ref{S:discuss}.

There are two available adjustments to account for the potentially empty realizations of the random set $\YY_{x_{n+1}}^n(\S)$.  The first, and probably most intuitive, is {\em conditioning} on the event that the random set is non-empty, which happens to be equivalent to Dempster's rule of combination \citep[e.g.,][Chap.~3]{shafer1976mathematical}.  For example, the post-conditioning plausibility contour is given by 
\[ y_{n+1} \mapsto \R_n\{\YY_{x_{n+1}}^n(\U) \ni y_{n+1} \mid \YY_{x_{n+1}}^n(\U) \neq \varnothing\}. \]
It is easy to see that conditioning simply rescales the original plausibility contour, making it larger at each $y_{n+1} \in \YY$.  Clearly, if the unadjusted probabilistic predictor is valid, then this conditioning adjustment---which only inflates its plausibility contour values---cannot fail to be valid.  This inflation does, however, suggest a potential loss of efficiency, e.g., larger prediction sets in \eqref{eq:pred.set}.


The second adjustment strategy, designed to preserve validity without sacrificing efficiency, is based on a suitable {\em stretching} of the original random set; see, e.g., \citet{leafliu2012}. Roughly, those $\U$ such that $\YY_{x_{n+1}}^n(\U) = \varnothing$ correspond to ``conflict cases,'' and Dempster's conditioning rule simply removes these conflict cases and renormalizes the $\U$-probabilities.  As an alternative, \citet{leafliu2012} suggested to stretch those conflict $\U$ cases just enough so that $\YY_{x_{n+1}}^n(\U)$ is non-empty.  Their formulation was in the context of inference under non-trivial parameter constraints, but here we apply this to classification.  


Start by defining the set 
\begin{equation}
\label{eq:UU}
\UU_{x_{n+1}}^n = \bigcup_{y_{n+1} \in \YY} \bigl\{r\bigl(T_{n+1}(z^n,z_{n+1})\bigr)\bigr\} \subseteq \I_{n+1}.
\end{equation}
There are only finitely many $y_{n+1}$ values, and the set $\UU_{x_{n+1}}^n$ defined above is just the collection of ranks that are possible for the given $Z^n$ and $x_{n+1}$.   
Note that $\YY_{x_{n+1}}^n(\U)$ is empty if and only if $\U$ has empty intersection with $\UU_{x_{n+1}}^n$.  Therefore, the conflict cases mentioned above can be
alternatively defined as realizations of $\U$ that have empty intersection with $\UU_{x_{n+1}}^n$. This conflicting situation can be avoided if, instead of throwing out the conflict $\U$, we stretch it to a suitable $\U_e$, with $e \geq 0$ a stretching parameter that controls how far $\U$ is stretched toward $\UU_{x_{n+1}}^n$.  In particular, we take 
\[\U_e = \{1, 2, \ldots, U' + e\}, \quad U' \sim \unif(\I_{n+1}).\]
Following \citet{leafliu2012}, the parameter $e$ is chosen as the smallest value at which the intersection of $\U_e$ and $\UU_{x_{n+1}}^n$ is non-empty, i.e.,
\begin{align*}
\hat e &= \min\{e: \U_e \cap \UU_{x_{n+1}}^n \neq \varnothing\} = \begin{cases}
\min\UU_{x_{n+1}}^n - U' & \text{if $U' < \min \UU_{x_{n+1}}^n$} \\
0 & \text{otherwise}.
\end{cases}
\end{align*}
Consequently, $\U_{\hat e}$ would be
\begin{align*}
\U_{\hat e} &= \begin{cases}
\{1, 2, \ldots, \min\UU_{x_{n+1}}^n\} & \text{if $U' < \min \UU_{x_{n+1}}^n$} \\
\{1, 2, \ldots, U'\} & \text{otherwise}.
\end{cases}
\end{align*}
In summary, in the stretching IM, the IM's original random set output $\YY_{x_{n+1}}^n(\U)$ is replaced with $\YY_{x_{n+1}}^n(\U_{\hat e})$, and its guaranteed non-emptiness makes the probabilistic predictor derived from it valid.  It is also more efficient than conditioning since it avoids globally inflating the plausibility contour via renormalization, as the following example highlights.



For illustration, consider the data in Table~\ref{tab:data}, taken from \citet[][p.~304]{agresti2003categorical}, describing the primary food choices and lengths of $n=39$ male alligators caught in Lake George, Florida. Assume the 40th caught alligator is two meters long, i.e., $X_{n+1} = 2$. The goal is to predict $Y_{n+1}$, its primary food choice.  Note that
\begin{equation}
\label{eq:YY_example}
\YY_{x_{n+1}}^n(\U) = \begin{cases}
\{I\} & \text{with probability 0.1} \\
\{I, F\} & \text{with probability 0.2} \\
\{I, F, O\} & \text{with probability 0.3} \\
\varnothing & \text{with probability 0.4.}
\end{cases}
\end{equation}
The corresponding
plausibility contour, as given in \eqref{eq:contour}, is represented by the solid lines in Figure~\ref{fig:IFO}(a). By
thresholding it at any $\alpha>0.6$ we obtain $100(1-\alpha)$\% prediction sets that are empty, which is undesirable.

\begin{table}[t]
\centering
\begin{tabular}{c c | c c | c c}
\hline
Length (m) & Choice & Length (m) & Choice  & Length (m) & Choice \\
\hline
1.30 & I &  1.65 & I &  2.03 & F   \\
1.32 & F &   1.65 & F &  2.31 & F  \\
1.32 & F &   1.68 & F  & 2.36 & F\\
1.40 & F &   1.70 & I  & 2.46 & F\\
1.42 & I &   1.73 & O  & 3.25 & O\\
1.42 & F &   1.78 & F  & 3.28 & O\\\
1.47 & I &   1.78 & O  & 3.33 & F \\
1.47 & F &   1.80 & F & 3.56 & F\\
1.50 & I &   1.85 & F & 3.58 & F\\
1.52 & I &   1.93 & I  & 3.66 & F\\
1.63 & I &   1.93 & F  & 3.68 & O\\
1.65 & O &   1.98 & I  & 3.71 & F\\
1.65 & O &   2.03 & F  & 3.89 & F\\
\hline
\end{tabular}
\caption{Primary food choice (I, invertebrates; F, fish; O, other) and lengths (in meters) for $n=39$ male alligators  \citep[][p.~304]{agresti2003categorical}.}
\label{tab:data}
\end{table}

The plausibility contour conditioned on \eqref{eq:YY_example} $\neq \varnothing$ is easy to evaluate, and is represented by the dashed lines in Figure~\ref{fig:IFO}(a).
To calculate the plausibility contour under the stretching approach, we obtain, after some calculations, $\UU_{x_{n+1}}^{n} = \{17, 21, 29\}$. As $\min\UU_{x_{n+1}}^{n}$ = 17,
\begin{align*}
\U_{\hat e} &= \begin{cases}
\{1, 2, \ldots, 17\} & \text{if $U' < 17$} \\
\{1, 2, \ldots, U'\} & \text{otherwise}.
\end{cases}
\end{align*}
where $U' \sim \unif(1, 2, \ldots, 40)$. Therefore,
\begin{align*}
\YY_{x_{n+1}}^{n}(\U_{\hat e}) &= \begin{cases}
\{I\} & \text{with probability 0.5} \\
\{I, F\} & \text{with probability 0.2} \\
\{I, F, O\} & \text{with probability 0.3,} \\
\end{cases}
\end{align*}
and the dotted lines in Figure~\ref{fig:IFO}(a) illustrate its corresponding plausibility contour. Note, first, that empty prediction sets are eliminated with both the conditioning and the stretching adjustments. Second, for any $\alpha$, the $100(1-\alpha)\%$ prediction sets derived from the stretching adjustment are no larger than the corresponding ones derived from the conditioning adjustment, which
indicates that the former is no less efficient than the latter.
Another way to see this is through the difference between the upper and lower probabilities derived by the respective probabilistic predictors. \citet{DEMPSTER2008365} referred to this gap as the ``don’t know'' probability. Of course, between two valid probabilistic predictors, the one with less ``don't know'' is preferred because it is more efficient.  Figure~\ref{fig:IFO}(b) shows the upper and lower probabilities for the singleton assertions $\{I\}$, $\{O\}$ and $\{F\}$, for both strategies.  Clearly, stretching leads to a more efficient probabilistic predictor.



\begin{figure}[t]
\begin{center}
\subfigure[]{\scalebox{0.6}{\includegraphics{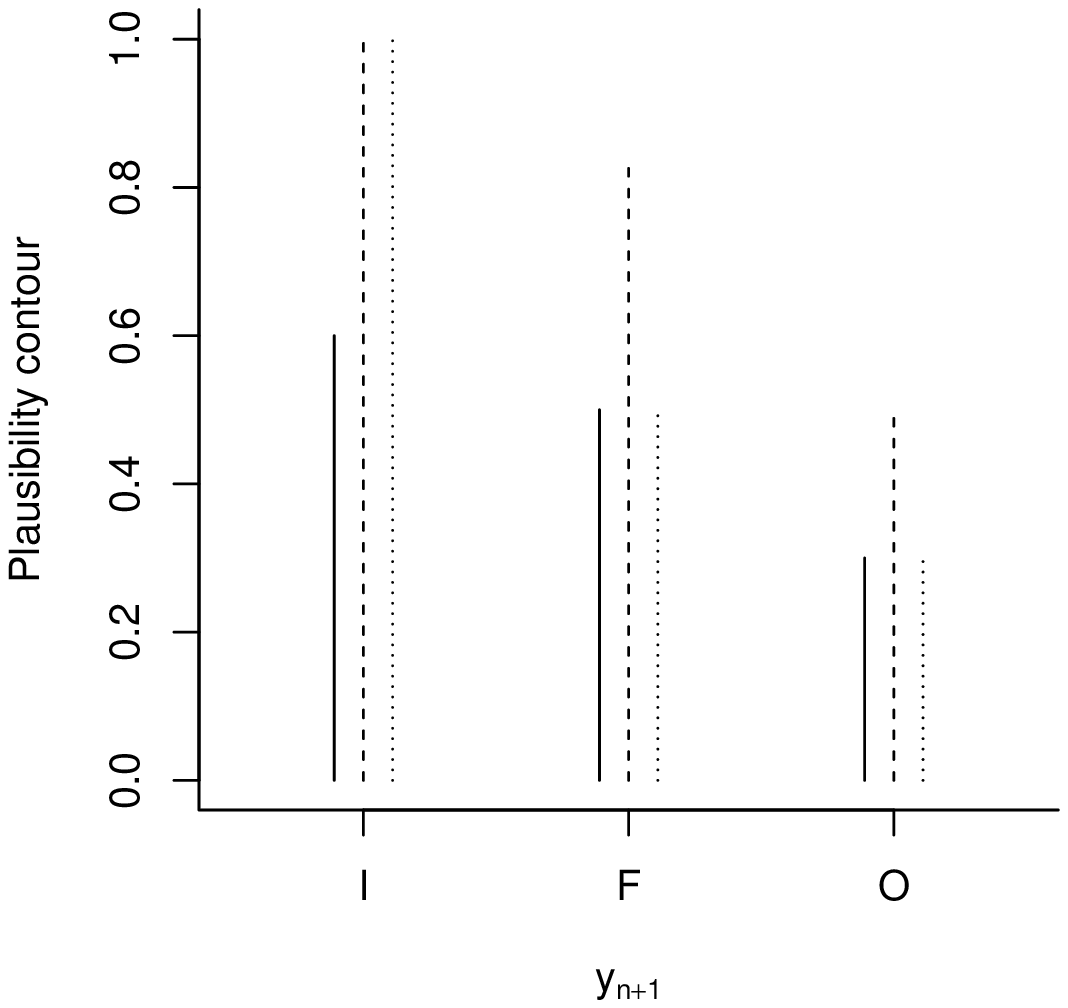}}}
\subfigure[]{\scalebox{0.6}{\includegraphics{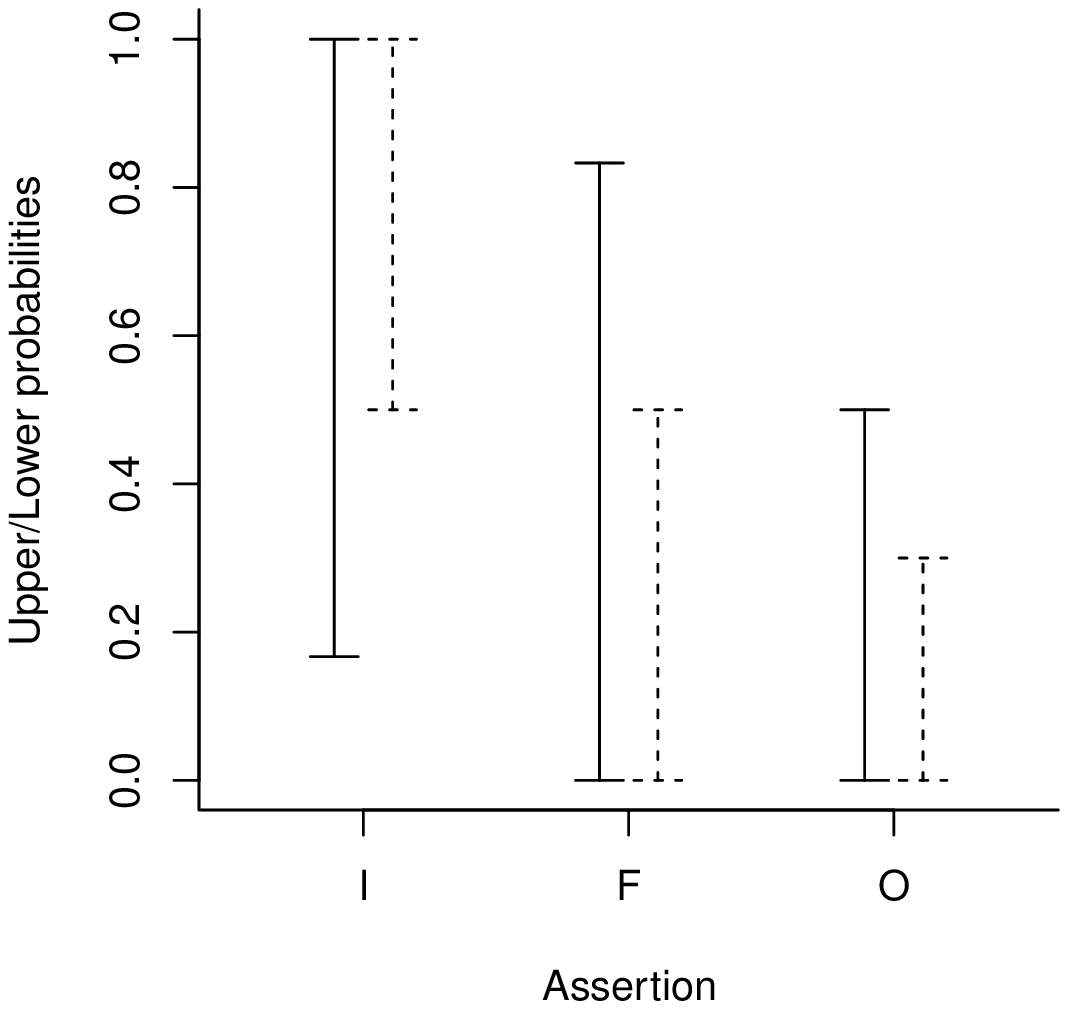}}}
\end{center}
\caption{Panel~(a): Plausibility contours in Equation~\eqref{eq:contour}, derived from an IM construction with no adjustment (solid lines), conditioning adjustment (dashed lines) and stretching adjustment (dotted lines).  Panel~(b): Upper and lower probabilities for the singleton assertions $\{I\}$, $\{F\}$ and $\{O\}$ derived from an IM construction with the conditioning adjustment (solid lines) and the stretching adjustment (dashed lines). These predictions are based on a new alligator of length $x_{n+1}=2$ meters.}
\label{fig:IFO}
\end{figure}

To further see this gain in efficiency we consider the {\em Glass Identification} data set from the USA Forensic Science Service, available in the UCI Machine Learning Repository \citep{Dua:2019}.\footnote{\url{https://archive.ics.uci.edu/ml/datasets/glass+identification}} It has 10 attributes associated with 214 glasses. The type of glass, a categorical variable---with six categories, including ``containers'' and ``headlamps''---is the response variable. The nine remaining variables, which describe the oxide content, i.e., Na, Fe, K, etc., are the explanatory variables. Classification of types of glass is relevant in criminology  applications, where glass fragments left at the scene of the crime may be important evidence if correctly identified. To evaluate the performance in classifying glass fragments, we randomly split the data in half and train both the conditioning and stretching strategies in the first half, with $\Psi$ function as in \eqref{eq:nonconform_class}. For further comparison, we also train a Bayesian multinomial regression model with default, non-informative priors on the parameters.\footnote{The {\em bamlss} R package \citep{Bayes_mult} was used to run the Bayesian analysis.} 
Figure~\ref{fig:dens_plaus} plots the distribution functions of the corresponding plausibility contours for the responses in the second half of the data. As expected, uniform validity in \eqref{eq:uvalid.alt} fails for the Bayesian solution and holds for both IM solutions, with the one based on stretching being more efficient.  Of course, not being uniformly valid does not imply that the Bayesian prediction set will not achieve the nominal coverage, but we can check this directly.  Table~\ref{tab:covlen} shows the empirical coverage probabilities and the average sizes (cardinality) of 95\% prediction sets for the responses in the second half of the data.  Clearly, the Bayes approach does not provide valid prediction sets. 


\begin{figure}[t]
\begin{center}
\scalebox{0.7}{\includegraphics{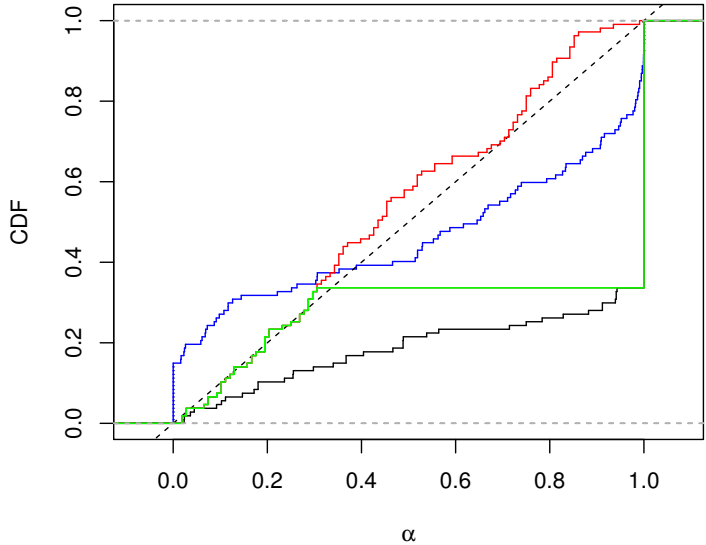}}
\end{center}
\caption{Distribution function for plausibility contours derived from an IM construction, with conditioning (black), stretching (green) and no adjustment (red), and a Bayesian
multinomial regression model (blue).}
\label{fig:dens_plaus}
\end{figure}

\begin{table}[t]
\centering
\begin{tabular}{c c c}
\hline
Strategy & Coverage & Size \\
\hline
Conditioning & 0.96 & 3.07  \\
Stretching & 0.96 & 2.73  \\
Bayes & 0.80 & 1.71 \\
\hline
\end{tabular}
\caption{Coverage probabilities and average size of 95\% prediction sets in \eqref{eq:pred.set} derived from an IM construction, with conditioning and stretching adjustment, and a Bayesian multinomial regression model. }
\label{tab:covlen}
\end{table}

Recall from Section~\ref{S:Regression} that the probabilistic predictor derived from the ``conformal + consonance'' construction is uniformly valid according to Definition~\ref{def:uvalid}, given that the conformal transducer $\pi_x^n$ satisfies \eqref{eq:sup}. In regression problems, this condition follows naturally from the continuity of $Y$, and the derived probabilistic predictor is equivalent to the one that would be obtained from an IM construction (assuming both use the same $\Psi$ function). In classification problems, however, \eqref{eq:sup} may not hold because $Y$ is discrete. This implies the ``conformal + consonance'' cannot be applied directly without some adjustment. This is not surprising given that similar adjustments were needed in the IM construction discussed above too. To better see this, note that the distribution function for the conformal transducer is also shown in Figure~\ref{fig:dens_plaus}. The need of an adjustment is evident, as uniform validity fails and, consequently, the derived conformal prediction intervals obtained through \eqref{eq:pred.set} would not be calibrated for certain choices of $\alpha$.

A natural adjustment is to force the conformal transducer to attain the value 1. 
Consider the following two adjusted conformal transducers:
\[\dot{\pi}_{x}^{n}(y) = \frac{\pi_{x}^{n}(y)}{\max_{y}\pi_{x}^{n}(y)},\]
and 
\begin{align*}
\ddot{\pi}_{x}^{n}(y) &= \begin{cases}
1 & \text{if $y= \hat y$}, \\
\pi_{x}^{n}(y) & \text{otherwise}, \\
\end{cases}
\end{align*}
where $\hat y = \arg\max_y \pi_{x}^n(y)$ and $y \in \YY$. In words,
$\dot{\pi}_{x}^{n}(y)$ takes the conformal transducers for the different $y \in \YY$ and divide them by their maximum, and $\ddot{\pi}_{x}^{n}(y)$ maintains all the conformal transducer values except for its maximum, which is assigned the value 1. That both adjusted transducers reach the value 1 makes the probabilistic predictors derived by them, through \eqref{eq:cons}, uniformly valid in the sense of Definition~\ref{def:uvalid}. It is also easy to see that these probabilistic predictors obtained from $\dot{\pi}_{x}^{n}(y)$ and $\ddot{\pi}_{x}^{n}(y)$  are equivalent to the ones derived from the IM construction with, respectively, the conditioning and the stretching adjustments. This shows that forcing consonance of the conformal transducer is not an ad hoc strategy; it is justified by the corresponding operations on random sets.  Moreover, in light of this connection to the IM's random set adjustments, we find that the second adjustment to the conformal predictor, i.e., setting the maximum value equal to 1, is the more efficient adjustment.

\section{Conclusion}
\label{S:discuss}

Here we focused on the important problem of prediction in supervised learning applications with no model assumptions (except exchangeability). We presented a notion of prediction validity, one that goes beyond the usual coverage probability guarantees of prediction sets. This condition assures the reliability of the degrees of belief, obtained from a imprecise probability distribution, assigned to all relevant assertions about the yet-to-be-observed quantity of interest. We also showed that, by following a new variation on the (generalized) IM construction first presented in \citet{martin2015, MARTIN2018105}, this validity property can be easily achieved.  We also noted the connection between this new IM construction and the conformal prediction strategy in, e.g., \citet{Vovk:2005}, and presented illustrations in both regression and classification settings. 
This connection is of paramount importance, as it implies that no new methodology is needed to achieve the (uniform) validity properties presented here. All that is needed is a possibilistic interpretation of the conformal prediction output.

Exchangeability was crucial to our IM construction, that is, without exchangeability, we cannot establish the distribution of the auxiliary variables.  While exchangeability is a relatively weak assumption compared to iid from a parametric family, there are, of course, situations where exchangeability is inappropriate, such as time series or spatial applications.  Work to develop conformal prediction methods in not-exactly-exchangeable settings is an active area of current research \citep[e.g.,][]{mao2020valid}, and it would be interesting to see what the IM perspective has to offer here. 

In Section~\ref{S:IM} we noted that the IM construction there leads naturally to a notion of {\em marginal} validity, which is different (and weaker) than the so-called {\em conditional} validity property.  While this is usually framed in the context of prediction sets, the corresponding definition in the context of probabilistic predictors is
\[ \prob\{\uPi_x^n(A) \leq \alpha, Y_{n+1} \in A \mid X_{n+1}=x\} \leq \alpha \quad \forall \; x, \]
and, of course, for all $(\alpha, n, A, \prob)$ as before.  Given the impossibility results in, e.g., \citet{leiwasserNPregression}, it seems unlikely that conditional validity can be achieved by any non-trivial probabilistic predictor.  Asymptotic conditional validity is possible, and some promising ideas are given in, e.g.,  \citet{chernozhukov2019distributional}.   

We mentioned in Section~\ref{S:Classification} that, surprisingly, empty random sets may have some practical value.  This concerns the so-called {\em open-} versus {\em closed-world} view of the prediction problem.  If the world is closed in the sense that all the possible labels are known, then it makes sense to remove the empty set cases and, hence, force consonance.  However, if the world is open in the sense that other labels are possible, then the empty set realization is an indication that the new object being classified may be of previously-unknown type, which itself is valuable information.  How this open-world view can be captured by the IM framework developed here remains an open question.

\section*{Acknowledgments}

The authors thank the reviewers from the conference proceedings and journal submissions for their valuable feedback, and the {\em IJAR} guest editors---Andr\'es Cano, Jasper De Bock, and Enrique Miranda---for the invitation to contribute to the special journal issue.  This work is partially supported by the U.S.~National Science Foundation, grants DMS--1811802 and SES--2051225.

\bibliographystyle{apalike}
\bibliography{literature}

\end{document}